\documentclass[11pt]{article}

\usepackage{amsmath,amsthm,amssymb,xypic}
\usepackage{hyperref}

\theoremstyle{plain}
    \newtheorem{theorem}{Theorem}[section]
    \newtheorem{lemma}[theorem]{Lemma}
    
    \newtheorem{proposition}[theorem]{Proposition}
    
     \newtheorem{question}[theorem]{Question}

 \theoremstyle{definition}
    \newtheorem{definition}[theorem]{Definition}
    
    \newtheorem{example}[theorem]{Example}

    \newtheorem{remark}[theorem]{Remark}

\theoremstyle{remark}

\numberwithin{equation}{section}

 \DeclareMathOperator{\Tr}{Tr}
 \DeclareMathOperator{\tr}{tr}

\DeclareMathOperator{\Ad}{Ad}
\DeclareMathOperator{\ad}{ad}

\DeclareMathOperator{\End}{End}

\DeclareMathOperator{\sgn}{sgn}

\DeclareMathOperator{\mat}{mat}

\DeclareMathOperator{\Pf}{Pf}

\DeclareMathOperator{\OO}{O}
\DeclareMathOperator{\SO}{SO}

\DeclareMathOperator{\GL}{GL}

\DeclareMathOperator{\U}{U}

         \DeclareMathOperator{\supp}{supp}

\DeclareMathOperator{\vol}{vol}
\DeclareMathOperator{\WF}{WF}

\DeclareMathOperator{\Real}{Re}
\DeclareMathOperator{\Imag}{Im}

\begin{document}

    \newcommand{\R}{\mathbb{R}}
    \newcommand{\C}{\mathbb{C}}
    \newcommand{\N}{\mathbb{N}}
    \newcommand{\Z}{\mathbb{Z}}
    \newcommand{\Q}{\mathbb{Q}}
    \newcommand{\bT}{\mathbb{T}}
    \newcommand{\bP}{\mathbb{P}}

\newcommand{\kg}{\mathfrak{g}}
\newcommand{\ka}{\mathfrak{a}}
\newcommand{\kb}{\mathfrak{b}}
\newcommand{\kk}{\mathfrak{k}}
\newcommand{\kt}{\mathfrak{t}}
\newcommand{\kp}{\mathfrak{p}}
\newcommand{\km}{\mathfrak{m}}
\newcommand{\kh}{\mathfrak{h}}
\newcommand{\kso}{\mathfrak{so}}

\newcommand{\cA}{\mathcal{A}}
\newcommand{\cE}{\mathcal{E}}
\newcommand{\calL}{\mathcal{L}}
\newcommand{\calH}{\mathcal{H}}
\newcommand{\cO}{\mathcal{O}}
\newcommand{\cB}{\mathcal{B}}
\newcommand{\cK}{\mathcal{K}}
\newcommand{\cP}{\mathcal{P}}
\newcommand{\cN}{\mathcal{N}}
\newcommand{\calD}{\mathcal{D}}
\newcommand{\cC}{\mathcal{C}}
\newcommand{\calS}{\mathcal{S}}
\newcommand{\cM}{\mathcal{M}}
\newcommand{\cU}{\mathcal{U}}
\newcommand{\cT}{\mathcal{T}}

\newcommand{\Rnz}{\R \setminus \{0\}}

\newcommand{\Bigwedge}{\textstyle{\bigwedge}}

\newcommand{\beq}[1]{\begin{equation} \label{#1}}
\newcommand{\eeq}{\end{equation}}

\newcommand{\ddt}{\left. \frac{d}{dt}\right|_{t=0}}
\newcommand{\mattwo}[4]{
\left( \begin{array}{cc}
#1 & #2 \\ #3 & #4
\end{array}
\right)
}

\newcommand{\dds}{\left. \frac{d}{ds}\right|_{s=0}}

\title{A Ruelle dynamical zeta function for equivariant flows}

\author{Peter Hochs\footnote{Radboud University, \texttt{p.hochs@math.ru.nl}} {}
and Hemanth Saratchandran\footnote{The University of Adelaide, \texttt{hemanth.saratchandran@adelaide.edu.au}}}

%




\date{\today}

\maketitle

\begin{abstract}
 For proper group actions on smooth manifolds, with compact quotients, we define an equivariant version of the Ruelle dynamical $\zeta$-function for equivariant flows satisfying a nondegeneracy condition.
  The construction is based on an equivariant generalisation of Guillemin's trace formula, obtained in a companion paper. This formula implies several properties of the equivariant Ruelle $\zeta$-function. We 
 ask the question in what situations an equivariant generalisation of Fried's conjecture holds, relating the equivariant Ruelle $\zeta$-function to equivariant analytic torsion. We compute the equivariant Ruelle $\zeta$-function in several examples, including examples where the classical Ruelle $\zeta$-function is not defined. The equivariant Fried conjecture holds in the examples where the condition of the conjecture (vanishing of the kernel of the Laplacian) is satisfied.
\end{abstract}

\tableofcontents

\section{Introduction}

The Ruelle dynamical $\zeta$-function is an expression involving the lengths of closed flow curves of a flow on a compact manifold. The Fried conjecture states that, for a reasonable class of flows, the value of this function at zero is equal to the analytic torsion of the manifold. In this paper, we construct an equivariant generalisation of the Ruelle $\zeta$-function for proper, cocompact group actions. We give examples where an equivariant generalisation of the Fried conjecture does and does not hold, and ask the question under what hypotheses this generalisation is true.

\subsection*{Background and motivation}

Let $M$ be a compact manifold, and $\varphi$ a flow on $M$. Suppose that $\varphi$ is \emph{Anosov}, i.e.\ that the tangent bundle of $M$ decomposes into the image of the generating vector field  of $\varphi$, and two sub-bundles on which the derivative of $\varphi$ is exponentially contractive as time goes to plus or minus infinity, respectively. An example of an Anosov flow is
the geodesic flow on the unit sphere bundle of the tangent bundle of a compact
Riemannian manifold with negative sectional curvature.

If a flow curve $\gamma$ of $\varphi$ satisfies $\gamma(l) = \gamma(0)$, then the linearised Poincar\'e map
\begin{equation}\label{lin_poincare_map}
P_{\gamma} \colon T_{\gamma(0)}M/\R\gamma'(0) \to  T_{\gamma(0)}M/\R\gamma'(0),
\end{equation}
induced by $T_{\gamma(0)} \varphi_l$, 
does not have $1$ as an  eigenvalue. Therefore, $\sgn(\det (1-P_{\gamma}))$ is well-defined; we then say that the flow $\varphi$ is \emph{nondegenerate}.
Let $\nabla^F$ be a flat connection on a vector bundle $F \to M$. For $\gamma$ as above, let $\rho(\gamma)$ be the parallel transport map with respect to 
$\nabla^F$ along $\gamma$ from $0$ to $l$. Let
\beq{eq def prim per intro}
 T^{\#}_{\gamma}:= \min\{t>0; \gamma(t) = \gamma(0)\} 
\eeq
be the primitive period of $\gamma$. Then the \emph{Ruelle dynamical $\zeta$-function} for $\varphi$ with respect to $\nabla^F$ is defined by
\beq{eq Ruelle intro}
 R_{\varphi, \nabla^F}(\sigma) := \exp\left( \sum_{\gamma} \frac{e^{-\sigma l}}{l} \sgn(\det (1-P_{\gamma})) \tr (\rho(\gamma)) T^{\#}_{\gamma} \right), 
\eeq
where the sum is over closed flow curves $\gamma$, and $l$ is the period of such a curve $\gamma$. (See \cite{Ruelle76a} for the version introduced by Ruelle, or Definition 2.7 in \cite{Shen21} for a more general definition.) For Anosov flows, the set of such curves is discrete, and the sum on the right in \eqref{eq Ruelle intro} converges for $\sigma \in \C$ with large real part. This follows from an exponential bound on the number of flow curves with periods  less than a given value; see (8) in \cite{Margulis} or Lemma 2.2 in \cite{DZ16}.

Let $\Delta_F := (\nabla^F)^* \nabla^F +  \nabla^F (\nabla^F)^*$ be the Laplacian on differential forms twisted by $F$ associated to $\nabla^F$. 
Let $N$ be the number operator on such forms, equal to $p$ times the identity on $p$-forms. Let $P$ be the orthogonal projection onto the kernel of $\Delta_F$. The
Ray--Singer \emph{analytic torsion} of $\nabla^F$ is defined by
\[
T(\nabla^F) := \exp\left( \frac{-1}{2}\left. \frac{d}{ds}\right|_{s=0} \frac{1}{\Gamma(s)} \int_0^{\infty} t^{s-1} \Tr( (-1)^Ne^{-t\Delta_F} - P)\, dt \right).
\]

The \emph{Fried conjecture} states \cite{Fried87} that if $\Delta_F$ has trivial kernel, then $R_{\varphi, \nabla^F}$ has a meromorphic extension to $\C$ that is regular at zero, and
\beq{eq Fried intro}
|R_{\varphi, \nabla^F}(0)| = T(\nabla^F).
\eeq
This was proved by Fried for the geodesic flow on hyperbolic manifolds \cite{Fried86}, and has since been proved in several other cases, see for example \cite{BismutHypo, DGRS20, MS91, Muller21, SM93, SM96, Shen18, SY21, SY17,  Wotzke08, Yamaguchi22}.  This includes versions for non-Anosov flows, such as geodesic flows on unit sphere bundles of compact manifolds whose sectional curvature is non-positive, but not everywhere negative. For such flows a more general version of the definition of the Ruelle $\zeta$-function is used, see for example Definition 2.7 in \cite{Shen21}. 
The Fried conjecture does not hold in general, 
see the end of Section 2 in \cite{Shen21}, and \cite{Schweitzer74, Wilson66}.

In this paper, we define a Ruelle dynamical $\zeta$-function for flows $\varphi$ satisfying an equivariant nondegeneracy condition, on manifolds $M$ with proper actions by groups $G$ such that $M/G$ is compact.  The basic  idea is to replace flow curves $\gamma$ satisfying $\gamma(l) = \gamma(0)$ by flow curves  $\gamma$ satisfying 
\beq{eq g periodic}
\gamma(l) = g\gamma(0)
\eeq
for a fixed element $g \in G$. 
We study several properties of such a $\zeta$-function, and verify an equivariant version of the Fried conjecture in some examples.

\subsection*{Construction and results}

Let $M$, $\varphi$ and $\nabla^F$ be as before, but in this section we don't assume $M$ is compact and that $\varphi$ is Anosov. Suppose that a unimodular, locally compact group $G$ acts properly on $M$ and that $M/G$ is compact. Suppose $\varphi$ is $G$-equivariant, and that the
group action lifts to $F$ and preserves $\nabla^F$. Fix an element $g \in G$ such that the quotient of $G$ by the centraliser of $g$ has a nonzero $G$-invariant Borel measure. The Anosov or weaker nondegeneracy condition on $\varphi$ is replaced by the assumption that for all flow curves $\gamma$ such that \eqref{eq g periodic} holds, the map
\[
T_{\gamma(0)} \varphi_l \circ g^{-1}\colon T_{\gamma(0)}M \to T_{\gamma(0)}M
\]
has one-dimensional eigenspace for eigenvalue $1$ (spanned by $\gamma'(0)$). Then we call the flow \emph{$g$-nondegenerate}.

As mentioned, the idea behind the definition of the equivariant Ruelle $\zeta$-function $R^g_{\varphi, \nabla^F}$ is to replace closed flow curves in \eqref{eq Ruelle intro} by flow curves $\gamma$ satisfying \eqref{eq g periodic}. We now also allow negative $l$, because in some cases \eqref{eq g periodic} may not hold for any positive $l$, see Example 2 below Remark \ref{rem P rho dep l}. More importantly, allowing negative $l$ seems to be essential for obtaining
a natural equivariant generalisation of Fried's conjecture; see Remarks \ref{rem Fried neg l} and \ref{rem neg l R}.

 There are straightforward generalisations of the maps $P_{\gamma}$ (defined in  \eqref{lin_poincare_map})  and $\rho(\gamma)$ to such flow curves.
The appropriate generalisation of the primitive period $T^{\#}_{\gamma}$ in \eqref{eq def prim per intro} is less obvious, however, particularly for non-discrete groups. This factor  
$T^{\#}_{\gamma}$
is replaced by a certain integral involving a cutoff function $\chi \in C^{\infty}_c(M)$, see \eqref{eq cutoff fn} and \eqref{eq def Tgamma}. The final definition of the equivariant Ruelle $\zeta$-function also involves an integral over the conjugacy class of $g$, making it conjugation-invariant. See Definition \ref{def equivar Ruelle} for details.

We obtain an expression for the equivariant Ruelle $\zeta$-function in terms of a distributional \emph{flat $g$-trace} $\Tr_g^{\flat}$ introduced in \cite{HS25a}, see Definition \ref{def flat g trace}. With slight abuse of notation (treating distributions as functions), Theorem \ref{thm cont R} is the equality
\beq{eq fixed pt intro}
 R_{\varphi, \nabla^F}^g(\sigma) = \exp\left( - \int_{\Rnz}
 \Tr_g^{\flat}\bigl((-1)^N N \varphi_t^*\bigr)\frac{e^{-\sigma |t|}}{|t|}\, dt \right).
\eeq
The proof is based on Theorem \ref{prop fixed pt gen}, an equivariant generalisation of Guillemin's trace formula \cite{Guillemin77}, obtained in \cite{HS25a}.

Applications of and motivation for Theorem \ref{thm cont R} include the following.
\begin{enumerate}
\item It  implies that the equivariant Ruelle $\zeta$-function is independent of the cutoff function $\chi$ used; see Proposition \ref{prop R indep chi}.
\item It implies a decomposition of the classical Ruelle $\zeta$-function for a flow on a compact manifold  into equivariant Ruelle $\zeta$-functions for the action by the fundamental group of the manifold on its universal cover; see Proposition \ref{prop R XM}.
\item As in the classical case,  Theorem \ref{thm cont R} may be a basis for proofs of further properties of the equivariant Ruelle $\zeta$-function, such as meromorphic continuation (see Section 2.2 in \cite{GLP13}, or \cite{DZ16}) and homotopy invariance (see Theorem 2 in \cite{DGRS20}).
\item As in the classical case, the expression \eqref{eq fixed pt intro} for the equivariant Ruelle $\zeta$-function suggests similarities with equivariant analytic torsion; see Remark \ref{rem Ruelle torsion det}.
\end{enumerate}

Replacing the left hand side of \eqref{eq Fried intro} by the equivariant Ruelle $\zeta$-function (with no absolute value), and the right hand side by the equivariant analytic torsion $T_g(\nabla^F)$ of $\nabla^F$ \cite{HS22a}, we obtain an equivariant version of Fried's conjecture:
\beq{eq g Fried intro}
R^g_{\varphi, \nabla^F}(0) = T_g(\nabla^F)^2.
\eeq 
See Question \ref{q Fried}. If $M$ is compact and $G$ is trivial, then this reduces to the classical Fried conjecture \eqref{eq Fried intro}. We do not expect \eqref{eq g Fried intro} to hold in general, and only ask the question under what conditions it may be expected to hold. Analogously to  the non-equivariant case, vanishing of the $L^2$-kernel of $\Delta_F$ is a condition. 

A relevant class of examples where the equivariant Fried conjecture does not hold is the action by the fundamental group of a compact manifold on its universal cover. Already in Fried's proof of his conjecture for hyperbolic manifolds \cite{Fried86}, and also in later proofs of other cases (see the references below \eqref{eq Fried intro}), the strategy is to decompose both sides of \eqref{eq Fried intro} into factors corresponding to conjugacy classes in the fundamental group. But only the total expressions are equal, not the factors for individual conjugacy classes. This is easiest to see for the trivial conjugacy class. Then the corresponding factor in the expression for analytic torsion is $L^2$-analytic torsion \cite{Lott92c, Mathai92} and is generally different from $1$, whereas the corresponding factor for the Ruelle $\zeta$-function is always $1$.

In several other cases, however, we compute the equivariant Ruelle $\zeta$-function explicitly, and find that its value at zero equals equivariant analytic torsion. These cases are:
\begin{enumerate}
\item the real line acting on itself, see Lemma \ref{lem Ruelle R};
\item the circle acting on itself, see Lemmas \ref{lem Ruelle S1} and \ref{lem Ruelle S1 e};
\item certain discrete subgroups of the Euclidean motion group acting on the unit sphere bundle of $\R^n$, with geodesic flow; see Proposition \ref{prop Fried Eucl}.
\end{enumerate}
For geodesic flow on the unit sphere bundle of the $n$-sphere, acted on by $\SO(n+1)$, we compute the equivariant Ruelle $\zeta$-function in Proposition \ref{prop Ruelle S2} for $n=2$ and in Proposition \ref{prop Ruelle S3} for $n=3$. In this case, 
the equivariant Ruelle $\zeta$-function does not extend meromorphically to zero, so the left hand side of 
\eqref{eq g Fried intro} is not defined. Now  the kernel of the Laplacian $\Delta_F$ is nonzero, so the conditions of the equivariant Fried conjecture are not satisfied.

For geodesic flow on $\R^n$, the classical Ruelle $\zeta$-function equals $1$, since there are no closed geodesics. For geodesic flow on spheres, this function is not defined, as there are uncountably many closed geodesics. It may be of interest that the equivariant Ruelle $\zeta$-function is defined, and nontrivial, in these cases, for suitable group elements.

As another relation between  the equivariant Ruelle $\zeta$-function and equivariant analytic torsion, we show that these quantities 
 in many cases have  the same behaviour with respect to restriction to subgroups; see Lemma \ref{lem Ruelle subgroup}.

These results, together with the suggestive similarities between the equivariant Ruelle $\zeta$-function and equivariant analytic  torsion in Remark \ref{rem Ruelle torsion det},  make us wonder if general conditions can be formulated under which the equivariant Fried conjecture is true. We plan to investigate this in the future.

In forthcoming work by the first author and Pirie, we obtain positive results on the equivariant Fried conjecture for the suspension flow of an equivariant isometry of a Riemannian manifold.

\subsection*{Acknowledgements}

We thank Jean-Michel Bismut, Bingxiao Liu, Thomas Schick, Shu Shen and Polyxeni Spilioti for helpful comments and discussions. PH is partially supported by the Australian Research Council, through Discovery Project DP200100729, and by NWO ENW-M grant OCENW.M.21.176.
HS was supported by the Australian Research Council, through grant FL170100020.

\section{The equivariant Ruelle dynamical $\zeta$-function}\label{sec def Ruelle}

We start by introducing the equivariant Ruelle dynamical $\zeta$-function, and state some of its basic properties. These include some relations to the classical Ruelle dynamical $\zeta$-function, Lemmas \ref{lem Rh cpt} and \ref{lem geodesic},  and Proposition \ref{prop R XM}. Using the equivariant version of analytic torsion from \cite{HS22a}, we state an equivariant version of the Fried conjecture as Question \ref{q Fried}.

In Section \ref{sec prop Ruelle}, we obtain some key properties 
of the equivariant Ruelle dynamical $\zeta$-function, including independence of the choices made in its definition. The proofs of these properties are based on the main result in \cite{HS25a}, an equivariant version of Guillemin's trace formula (Theorem \ref{prop fixed pt gen}).

\subsection{Definition of the equivariant Ruelle dynamical $\zeta$-function}\label{sec def def Ruelle}

We recall some notation, definitions and results from  Section 2 of \cite{HS25a}.

Let $M$ be a smooth manifold. Let $\varphi$ be the flow of a smooth vector field $u$ on $M$. Suppose that $u$ has no zeroes, and that its flow $\varphi$ is defined for all time. 
Let $G$ be a unimodular locally compact group, acting properly and cocompactly on  $M$.
Suppose that for all $t \in \R$, the flow map $\varphi_t\colon M \to M$ is $G$-equivariant.
This is true in the following special cases.
\begin{enumerate}
\item If $G$ acts properly and isometrically on a Riemannian manifold $X$, then geodesic flow on the unit sphere bundle $M = S(TX)$ is equivariant with respect to the lifted action of $G$. (See Lemma 2.10 in \cite{HS25a}.) \label{ex geod flow invar}
\item If $M$ is the universal cover of a compact manifold $X$, then  a flow on $X$ has a unique lift to a flow on $M$ that is equivariant with respect tot he action of $G = \pi_1(X)$. (See Example 2.11 in \cite{HS25a}.) 
\item If a compact Lie group $G$ acts locally freely on a compact manifold $X$, consider the diagonal action by $G$ on $M = X \times S(\kg)$, where $S(\kg) \subset \kg$ is the unit sphere for a $G$-invariant inner product. Then the flow $\varphi$ on $M$ given by
\[
\varphi_t(p,Y) = (\exp(tY)p, Y),
\]
for $t \in \R$, $p \in X$ and $Y \in S(\kg)$, satisfies the conditions stated above. (See Example 2.12 in \cite{HS25a}.) 
\end{enumerate}

Let $\chi \in C^{\infty}_c(M)$ be such that for all $m \in M$,
\beq{eq cutoff fn}
\int_G \chi(xm)\, dx = 1.
\eeq
Such a function exists because the action is proper and $M/G$ is compact. Furthermore, let $F \to M$ be a flat,  $G$-equivariant vector bundle, and $\nabla^F$ a flat, $G$-invariant connection on $F$. Parts (a)--(d) of the following definition are Definition 2.17 in \cite{HS25a}.
\begin{definition}\label{def length spectrum}
Let $x \in G$.
\begin{itemize}
\item[(a)]
The \emph{$x$-delocalised length spectrum} of $\varphi$ is
\beq{eq def Lg}
L_{x}(\varphi) := \{l \in \Rnz; \exists m \in M: \varphi_l(m) = x m\}.
\eeq
\end{itemize}
Let
 $l \in L_{x}(\varphi)$.
 \begin{itemize}
 \item[(b)]
A \emph{flow curve}  of $\varphi$ is a curve $\gamma\colon \R \to M$ such that $\gamma(t) = \varphi_t(\gamma(0))$ for all $t \in \R$. Let $\Gamma_l^x(\varphi)$ be the set of equivalence classes of flow curves $\gamma$ such that $\gamma(l) = x\gamma(0)$, where two such curves are equivalent if they have the same image (i.e.\ they are related by a time shift).
 \end{itemize}
 Let
 $\gamma \in \Gamma_l^{x}(\varphi)$.
 \begin{itemize}
 \item[(c)]
The  \emph{linearised $x$-delocalised Poincar\'e map}  of $\gamma$ is the map
\beq{eq def Poincare}
P^{x}_{\gamma} \colon T_{\gamma(0)}M/\R u(\gamma(0)) \to T_{\gamma(0)}M/\R u(\gamma(0))
\eeq
induced by $T_{\gamma(0)} (\varphi_{l}\circ x^{-1})$. 
\item[(d)] The \emph{$\chi$-primitive period} of $\gamma$ is
\beq{eq def Tgamma}
T_{\gamma} := \int_{I_{\gamma}} \chi(\gamma(s))\, ds
\eeq
(when convergent), 
where $I_{\gamma}$ is an interval such that $\gamma|_{I_{\gamma}}$ is a bijection onto the image of $\gamma$, modulo sets of measure zero. 
\item[(e)] 
For all $m \in M$ and $t \in \R$, let $\tau^{\nabla^F}_t \colon F_m \to F_{\varphi_t(m)}$ be
 parallel transport with respect to $\nabla^F$, along the flow curve from $m$ to $\varphi_t(m)$.  
The map
\beq{eq rho g}
 \rho_x(\gamma)\colon F_{\gamma(0)} \to F_{\gamma(0)}
 \eeq
is the parallel transport map  $\tau^{\nabla^F}_{-l}\colon F_{\gamma(0)} \to F_{\gamma(-l)}$ along $\gamma$ with respect to $\nabla^F$, composed with $x\colon F_{\gamma(-l)} \to F_{\gamma(0)}$.
\end{itemize}
\end{definition}


\begin{remark}
The $x$-delocalised length spectrum \eqref{eq def Lg} may be empty, for example when $x=e$ in the settings of Subsection \ref{sec class decomp} and Section \ref{sec Eucl}. Then we find that using nontrivial group elements still allows us to define a meaningful version of the Ruelle $\zeta$-function. 
\end{remark}

\begin{remark}\label{rem P rho dep l}
For a flow curve $\gamma \in \Gamma_l^g(\varphi)$, always defined on $\R$, the maps $P^x_{\gamma}$ and $\rho_x(\gamma)$ depend on $l$. 
\end{remark}

Examples 2.19, 2.21 and 2.22 in \cite{HS25a} are the following.
\begin{enumerate}
\item If $M = G = \R$, acting on itself by addition, and $\varphi$ is given by $\varphi_t(m) = m+t$, then
 $L_x(\varphi) = \{x\}$ for all nontrivial $x \in G$, 
 $T_{\gamma} = 1$ for the only flow curve $\gamma$ (modulo time shifts).
\item If $G$ is compact, and $\gamma$ is periodic  (see part (b) of Lemma \ref{lem conv Tgamma} for a sufficient condition), then $T_{\gamma} = T_{\gamma}^{\#}$, the primitive period of $\gamma$.
\item For the geodesic flow on the universal cover of a negatively curved manifold, the $\chi$-primitive period of a geodesic equals its primitive period in the usual sense, for well-chosen $\chi$. See the proof of Lemma \ref{lem geodesic}. 
\end{enumerate}


From now on, we fix $g \in G$.
%
\begin{definition}\label{def nondeg}
The flow $\varphi$ is \emph{$g$-nondegenerate} if 
the map $P_{\gamma}^g - 1_{T_{\gamma(0)}M/\R u(\gamma(0))}$ is invertible for all $l\in L_g(\varphi)$ and $\gamma \in \Gamma_l^g(\varphi)$.
\end{definition}
%
%
If $g=e$, then a Anosov flows are examples of $e$-nondegenerate flows; see Remark 2.24 in \cite{HS25a}.

The following two lemmas are Lemmas 2.25 and 2.26 in \cite{HS25a}.
\begin{lemma}\label{lem Gammagl discrete}
Suppose that  $\varphi$ is {$g$-nondegenerate} and
let  $l \in L_{g}(\varphi)$. 
%
Then the set $\Gamma^{g}_l(\varphi)$ is discrete, in the sense that it is countable, and the images of the curves in $\Gamma^{g}_l(\varphi)$ together form the closed subset  $M^{g^{-1}\varphi_l} \subset M$ of points fixed by $g^{-1} \circ \varphi_l$.
\end{lemma}
This lemma implies that  the set
\beq{eq Gamma hghinv}
 \Gamma_l^{hgh^{-1}}(\varphi) = h \cdot \Gamma_l^{g}(\varphi)
\eeq
  is discrete for all $l \in L_g(\varphi)$ if $\varphi$ is $g$-nondegenerate. 

\begin{lemma}\label{lem conv Tgamma}
Suppose that $\varphi$ is $g$-nondegenerate. Let $l \in L_g(\varphi)$ and $\gamma \in \Gamma^g_l(\varphi)$.
\begin{enumerate}
\item[(a)] The expression \eqref{eq def Tgamma} for $T_{\gamma}$ converges.
\item[(b)] If $M$ is compact, or $g$ lies in a compact subgroup of $G$, then $\gamma$ is periodic.
\end{enumerate}
\end{lemma}

Let $Z<G$ be the centraliser of $g$, and assume that $Z$ is unimodular. Then
$G/Z$ has a $G$-invariant measure $d(hZ)$. 
\begin{definition}\label{def equivar Ruelle}
Suppose that $L_{g}(\varphi)$ is countable, and that
$\varphi$ is $g$-nondegenerate.
If the following expression \eqref{eq Ruelle finite L} converges for $\sigma \in \C$ with large real part, then for such $\sigma$, the \emph{equivariant Ruelle dynamical $\zeta$-function} at $g$ is defined by
\begin{multline}\label{eq Ruelle finite L}
\log R_{\varphi, \nabla^F}^g(\sigma) = \\
\lim_{r \to \infty}
\int_{G/Z} \sum_{l \in L_{g}(\varphi) \cap [-r,r]}  \frac{ e^{-|l|\sigma}}{|l|} \sum_{\gamma \in \Gamma^{hgh^{-1}}_l(\varphi)} 
\sgn\bigl(\det(1-P_{\gamma}^{hgh^{-1}}) \bigr) \tr(\rho_{hgh^{-1}}(\gamma)) T_{\gamma}\, d(hZ).
\end{multline}
\end{definition}

\begin{remark} \label{rem indep gamma}
The numbers $\det(1-P_{\gamma}^{hgh^{-1}})$ and $\tr(\rho_{hgh^{-1}}(\gamma))$ in \eqref{eq Ruelle finite L} are independent of the choice of the representative $\gamma$ of a class in $\Gamma_l^{hgh^{-1}}(\varphi)$. Indeed, suppose that $ x\in G$,  $l \in L_x(\varphi)$ and  that $\gamma_1, \gamma_2 \in \tilde \Gamma_l^x(\varphi)$ 
have the same image. Then $\gamma_2(0) = \gamma_1(s)$ for some $s \in \R$, and
\[
\begin{split}
\rho_x(\gamma_2) &= \tau^{\nabla^F}_s \circ \rho_x(\gamma_1) \circ \tau^{\nabla^F}_{-s} \\
P_{\gamma_2}^x &= T_{\gamma_1(0)}\varphi_s \circ P_{\gamma_1}^x\circ T_{\gamma_1(0)}\varphi_{-s}.
\end{split}
\]
Furthermore,  the function $R_{\varphi, \nabla^F}^g$ does not depend on the choices made in the definition of $T_{\gamma}$, see Proposition \ref{prop R indep chi} below.
\end{remark}

\begin{remark}
The function $R_{\varphi, \nabla^F}^g$ only depends on the conjugacy class of $g$. 
\end{remark}

\begin{remark}
The equivariant Ruelle dynamical $\zeta$-function depends on the $G$-invariant, flat connection $\nabla^F$ on $F$. In the classical case, where $G = \{e\}$, such a flat vector bundle $F$ with a  flat connection  is given by a representation of the fundamental group of $M$ via the Riemann--Hilbert correspondence. In that case,  the  Ruelle dynamical $\zeta$-function is usually viewed as depending on such a representation.  
\end{remark}

\begin{remark}
In the general definition of the classical Ruelle dynamical $\zeta$-function, for possibly non-Anosov flows, an orbifold Euler characteristic appears. See e.g.\ Definitions 2.5 and 2.7 in \cite{Shen21}. In the case of flows that are $g$-nondegenerate that we consider here, this orbifold Euler characteristic equals $1$. This is because the relevant fixed point set quotiented out by the real line (or circle), is a point by Lemma \ref{lem Gammagl discrete}. Compare the cases of Anosov flows, or geodesic flows on manifolds with negative sectional curvature, in the classical case; see (1.9) in \cite{Shen18}.
We consider the $g$-nondegenerate case to obtain an expression for the equivariant Ruelle $\zeta$-function in terms of a distributional trace, see Theorem \ref{thm cont R}.
\end{remark}

\subsection{An equivariant Fried conjecture}

Suppose for now that $M$ is a compact, oriented Riemannian manifold, and $G$ is trivial. Let $\rho$ be a unitary representation of $\pi_1(M)$. Ray and Singer \cite{RS71} defined the \emph{analytic torsion} $T_{\rho}(M)$ of $M$, with respect to $\rho$.
In this setting, the Fried conjecture states that, for suitable classes of flows $\varphi$ on $M$, the classical Ruelle dynamical $\zeta$-function $R_{\varphi, \rho}$ extends meromorphically to the complex plane, is regular in a neighbourhood of  zero, and satisfies
\beq{eq Fried conj}
|R_{\varphi, \rho}(0)| = T_{\rho}(M).
\eeq
This conjecture has been proved in several cases, see for example 
 \cite{BismutHypo, DGRS20, Fried86, MS91, Muller21, SM93, SM96, Shen18, SY21,SY17,  Wotzke08, Yamaguchi22}, this includes versions for non-Anosov flows.
The equality \eqref{eq Fried conj} is known to not hold in general, see the end of Section 2 in \cite{Shen21} and \cite{Schweitzer74, Wilson66}.

In \cite{HS22a}, we defined an equivariant version of analytic torsion for proper group actions with compact quotients. It is therefore natural to ask under what conditions an equivariant version of \eqref{eq Fried conj} holds.

To state this question more explicitly, we return to the equivariant setting considered before. We now also assume that $M$ has an orientation and a $G$-invariant Riemannian metric, and that $F$ has a $G$-invariant Hermitian metric preserved by $\nabla^F$. 
Consider the Laplacian
\[
\Delta_F^p := (\nabla^F)^*\nabla^F + \nabla^F  (\nabla^F)^*\colon \Omega^p(M; F) \to \Omega^p(M; F), 
\]
where we view $\nabla^F$ as an operator on the space $ \Omega^*(M; F)$ of differential forms on $M$, twisted by $F$, and $(\nabla^F)^*$ is its formal adjoint, defined in the usual way in terms of  an extension of the Hodge $*$-operator extended to differential forms twisted by $F$. Let $P_p$ be the orthogonal projection from the Hilbert space of square integrable $p$-forms twisted by $F$ onto the kernel of $\Delta_F^p$ inside this space.

For a $G$-equivariant  bounded operator $T$ on the space of $L^2$-sections of a Hermitian $G$-vector bundle over $M$, we define its \emph{$g$-trace} as
\beq{eq g trace}
\Tr_g(T) := \int_{G/Z} \Tr(hgh^{-1} T\chi)\, d(hZ),
\eeq
if $T\chi$ is trace-class and the integral converges. This trace has been used in various places besides \cite{HS22a}, see for example \cite{HWW, HW2, Lott99, Wangwang}.

 Assuming that the following expressions converge, we write
\beq{eq curly T}
\cT_g(t) := \sum_{p=1}^{\dim(M)} (-1)^p p \Tr_g(e^{-t\Delta_F^p} - P_p)
\eeq
for $t>0$, and
\beq{eq def torsion}
T_g(\nabla^F) := \exp\left( -\frac{1}{2} \left. \frac{d}{ds}\right|_{s=0} \frac{1}{\Gamma(s)} \int_0^{
1} t^{s-1} \cT_g(t) \, dt
-\frac{1}{2}
\int_1^{
\infty} t^{-1}  \cT_g(t) \, dt \right).
\eeq
In the first term, we use meromorphic continuation from $s$ with large real part to $s=0$, again assuming the resulting expression is well-defined. Then $T_g(\nabla^F)$ is the \emph{equivariant analytic torsion} of $\nabla^F$ at $g$. For results on convergence and meromorphic extension of the integrals involved, see Section 3 of \cite{HS22a}. For independence of the Riemannian metric and other properties, see Sections 4 and 5 of \cite{HS22a}. For metric independence, a condition is that $\Delta_F^p$ has trivial kernel for all $p$. 

\begin{question}[Equivariant Fried conjecture]\label{q Fried}
Suppose that $M$ is odd-dimensional, and $\Delta_F^p$ has trivial $L^2$-kernel for all $p$. Under what further conditions does 
\begin{enumerate}
\item
the expression \eqref{eq Ruelle finite L} for $R^g_{\varphi, \nabla^F}(\sigma)$ converge when $\sigma$ has large real part;
\item this
have a meromorphic extension to the complex plane that is regular in a neighbourhood of zero; and
\item the equality
\beq{eq Fried}
R^g_{\varphi, \nabla^F}(0) = T_g(\nabla^F)^2
\eeq
hold?
\end{enumerate}
\end{question}

\begin{remark}
The square on the right hand side of \eqref{eq Fried}, and also the squares in
 Lemma \ref{lem Rh cpt} and Proposition \ref{prop R XM}, can be avoided if one changes the definition of the equivariant Ruelle dynamical $\zeta$-function by introducing a factor $1/2$ in front of the integral in \eqref{eq Ruelle finite L}. However, the resulting square root that is taken may be an obstruction to meromorphic continuation of this function. For this reason, we have chosen not to include this factor $1/2$.
\end{remark}

\begin{remark}
The definition of the equivariant Ruelle $\zeta$-function does not involve a Riemannian metric, but the definition of equivariant analytic torsion does. We expect equivariant analytic torsion to be independent of the choice of this metric, see Section 4 in \cite{HS22a} for results in this direction. Then the equality \eqref{eq Fried} for one $G$-invariant Riemannian  metric implies this equality for all such metrics. 
%
\end{remark}

\begin{example}\label{ex Fried class}
In the classical setting, where $G$ is trivial and $M$ is compact, the left hand side of \eqref{eq Fried} is the absolute value squared of the classical Ruelle dynamical $\zeta$-function at zero by Lemma \ref{lem Rh cpt}, while the right hand side is the square of classical Ray--Singer analytic torsion. So in this case, the equality \eqref{eq Fried} is the usual Fried conjecture \eqref{eq Fried conj}. Then the answer to Question \ref{q Fried} is \emph{yes} in cases where the classical Fried conjecture was proved.

Slightly more generally, if $G$, and hence $M$, is compact, and $g=e$, then \eqref{eq Fried} is equivalent to the classical Fried conjecture.
\end{example}

\begin{remark}\label{rem Fried neg l}
In the non-equivariant case, analytic torsion is a positive number. This is consistent with the absolute value on the left-hand side of \eqref{eq Fried conj}. In the equivariant case, analytic torsion need not be positive; see the examples in Sections \ref{sec ZR} and \ref{sec Eucl}.  Allowing negative $l$ in the definition of the equivariant Ruelle $\zeta$-function lets us formulate the equivariant Fried conjecture \eqref{eq Fried} without absolute values. For a version of the equivariant Ruelle $\zeta$-functions that only includes positive $l$, the analogue of \eqref{eq Fried} does not even hold with
absolute values on both sides in basic cases; see Remark \ref{rem neg l R} for an example.
\end{remark}

\begin{remark}
In Question \ref{q Fried}, we assume that $M$ is odd-dimensional, because the question then reduces to the classical Fried conjecture if $G$ is trivial, as in Example \ref{ex Fried class}. 

If $M$ is even-dimensional, then by Lemma \ref{lem Rh cpt}, the equivariant Ruelle $\zeta$-function for trivial groups becomes the square of the phase of the classical  Ruelle $\zeta$-function, rather than the square of its absolute value. It then follows that \eqref{eq Fried} does not reduce to the classical Fried conjecture.
For $M$ even-dimensional, as in the classical case, the equivariant analytic torsion of $M$ is trivial by Proposition 5.6 in \cite{HS22a}, so the case where $M$ is odd-dimensional is the most interesting. In the important example of the geodesic flow on the sphere bundle $M$ on a Riemannian manifold, the manifold $M$ is odd-dimensional.
\end{remark}

\begin{example}\label{ex Fried g=e}
In several proofs of the Fried conjecture in different settings, both sides of \eqref{eq Fried conj} are decomposed as sums of  terms corresponding to conjugacy classes in the fundamental group. The term in the decomposition of the Ruelle dynamical $\zeta$-function corresponding to a conjugacy class is an equivariant  Ruelle dynamical $\zeta$-function by Lemma \ref{lem geodesic}, in the case of the geodesic flow on a negatively curved manifold.  The terms on the two sides of the Fried conjecture corresponding to the same conjugacy class are not equal in general, even if the sums over all conjugacy classes are equal.

This is perhaps clearest for the trivial conjugacy class $\{e\}$. In the setting of Proposition \ref{prop R XM}, the number  $T_e(\nabla^{F_X})$ is the \emph{$L^2$-analytic torsion} of $X$ with respect to $\rho$ \cite{Lott92c, Mathai92}. This is generally different from $1$. However, the lift of a suitable flow on $X$ to $M$ may not have any closed loops, so that $R^e_{\varphi, \nabla^F}$ is the constant $1$. In fact, in Fried's proof of his conjecture in the case of hyperbolic manifolds \cite{Fried86},  this difference exactly compensates for a functional relation satisfied by certain $\zeta$-functions related to $R_{\varphi_X, \rho}$, see (FE) and (FE)' on page 531 of \cite{Fried86}.
In such cases, the answer to Question \ref{q Fried} is \emph{no} for the action by $G = \pi_1(X)$ on $M$, for $g=e$ and  at least one other conjugacy class.
\end{example}

For examples where  the answer to Question \ref{q Fried} is \emph{yes}, see Sections \ref{sec ZR} and \ref{sec Eucl}. For geodesic flow on the two- and three-spheres, discussed in Section \ref{sec Ruelle S2}, the equivariant Ruelle $\zeta$-function is defined for $\sigma$ with large real parts for suitable group elements, but does not continue meromorphically  to $\sigma = 0$. 
So 
the equality \eqref{eq Fried} is not defined, but in this case  the $L^2$-kernels of the operators $\Delta_F^p$ are nonzero. So this is not a counterexample to Question \ref{q Fried}. For more positive answers to Question \ref{q Fried}, in the case of the suspension flow of an equivariant isometry, see forthcoming work by the first author and Pirie.

\section{Properties of the equivariant Ruelle dynamical $\zeta$-function}\label{sec prop Ruelle}

\subsection{An equivariant Guillemin trace formula}

We recall the equivariant version of Guillemin's trace formula obtained in \cite{HS25a}. This is 
Theorem \ref{prop fixed pt gen}, which involves the distributional  \emph{flat $g$-trace} introduced in \cite{HS25a}.
We then use this to express the equivariant Ruelle dynamical $\zeta$-function in terms of this flat $g$-trace, see Theorem \ref{thm cont R} below. That expression is then used to prove independence of the Ruelle $\zeta$-function of choices made in its definition, and a relation with the classical Ruelle $\zeta$-function (Propositions \ref{prop R indep chi} and \ref{prop R XM}).

We recall the situation of Section 2.1 of \cite{HS25a}.
Let $E \to M$ be a $G$-equivariant vector bundle. We fix a $G$-invariant nonvanishing density, and use it to identify the space $\Gamma^{-\infty}(E)$ of generalised sections with the continuous dual of $\Gamma_c^{\infty}(E^*)$. (And to make similar idenfications for other vector bundles.) Let $\kappa \in \Gamma^{-\infty}(E \boxtimes E^*)$ be the Schwartz kernel of a $G$-equivariant, continuous linear operator $T \colon \Gamma_c^{\infty}(E)\to \Gamma^{\infty}(E)$. If $x \in G$, then we write $x\cdot \kappa$ for the Schwartz kernel of $x \circ T$. Let $p_2\colon M\times M \to M$ be the projection onto the second factor. Let $\Delta(M) \subset M \times M$ be the diagonal. We have a fibrewise trace $\tr$ of elements of  $\Gamma^{-\infty}(E \boxtimes E^*|_{\Delta(M)})$, see Definition 2.1 in \cite{HS25a}. Recall that we have chosen a cutoff function $\chi \in C^{\infty}_c(M)$ as in \eqref{eq cutoff fn}. 
\begin{definition}\label{def flat g trace}
The operator $T$ is \emph{flat $g$-trace class} if 
\begin{enumerate}
\item for all $h \in G$, the wave front set $\WF(hgh^{-1} \cdot \kappa)$ is disjoint from the conormal bundle to $\Delta(M) \subset M \times M$; and
\item the integral
\[
 \int_{G/Z} \bigl\langle \tr (hgh^{-1} \cdot  \kappa)|_{\Delta(M)}, (p_2^*f)|_{\Delta(M)} \bigr\rangle \, d(hZ)
\] 
converges for all $f \in C^{\infty}_c(M)$, and defines a distribution on $M$.
\end{enumerate} 
In that case, the \emph{flat $g$-trace} of $T$ is
\[
\Tr_g^{\flat}(T) :=  \int_{G/Z} \bigl\langle \tr (hgh^{-1} \cdot  \kappa)|_{\Delta(M)}, (p_2^*\chi)|_{\Delta(M)} \bigr\rangle \, d(hZ).
\]
\end{definition}
This definition is independent of $\chi$ by Proposition 2.9 in \cite{HS25a}.

Now let $\Phi$ be a $G$-equivariant flow in $E$ that lifts $\varphi$. If $\psi \in C^{\infty}_c(\Rnz)$, then we define the operator $\Phi_{\psi}^*$ on $\Gamma^{\infty}(E)$ by
\[
(\Phi_{\psi}^*s)(m) := \int_{\R} \psi(t) \Phi_{-t} s(\varphi_t(m))\, dt,
\]
for $s \in \Gamma^{\infty}(E)$ and $m \in M$. Let $A \in \End(E)$ be $G$-equivariant, and suppose that it commutes with $\Phi$.
If $A \Phi^*_{\psi}$ is flat $g$-trace class for $\psi \in  C^{\infty}_c(\Rnz)$, then we write
\beq{eq flat g trace flow}
\langle \Tr_g^{\flat}(A \Phi^*), \psi\rangle := \Tr_g^{\flat}(A \Phi^*_{\psi}).
\eeq
 If $x \in G$, $l \in L_x(\varphi)$ and $\gamma \in \Gamma_l^x(\varphi)$, then we write
\[
\tr(Ax\Phi_{-l}|_{\gamma}):= \tr(Ax(\Phi_{-l})_{\gamma(0)}).
\]

\begin{theorem}[Equivariant Guillemin trace formula]\label{prop fixed pt gen}
Suppose that the flow $\varphi$ is $g$-nondegenerate. Then the operator $A \Phi^*_{\psi}$ is flat $g$-trace class for all $\psi \in C^{\infty}_c(\Rnz)$. Furthermore, \eqref{eq flat g trace flow} defines a distribution on $\Rnz$, equal to
\beq{eq fixed pt Trb 2}
\Tr^{\flat}_g(A \Phi^*) = \int_{G/Z}  \sum_{l \in L_{g}(\varphi)} \sum_{\gamma \in \Gamma^{hgh^{-1}}_l(\varphi)} 
\frac{\tr\bigl(A hgh^{-1}\Phi_{-l}|_\gamma \bigr) }{|\det(1-P_{\gamma}^{hgh^{-1}})  |} T_{\gamma} \delta_{l}  \, d(hZ).
\eeq
Here $\delta_l$ is the Dirac $\delta$ distribution  at $l$. 
\end{theorem}
This result is Theorem 2.28 in \cite{HS25a}.

\subsection{An expression for the equivariant Ruellle dynamical $\zeta$-function}\label{sec Ruelle fixed pt}

We will use some basic linear algebra.
\begin{lemma}\label{lem tr det A}
Let $V$ be a finite-dimensional vector space, and $A \in \End(V)$. If $\dim(\ker(1-A)) = 1$, then
\beq{eq tr det}
\sum_{j=1}^{\dim(V)} (-1)^j j \tr(\wedge^j A) =- \det( (1-A)|_{V/\ker(1-A)}).
\eeq
\end{lemma}
\begin{proof}
For all $t \in \R$, 
\beq{eq tr det t}
\sum_{j=1}^{\dim(V)} (-1)^j  \tr(\wedge^j tA) = \det(1-tA).
\eeq
The derivative at $t=1$ of the left hand side is the left hand side of \eqref{eq tr det}. If $\lambda_1, \ldots, \lambda_{\dim(V)}$ are the eigenvalues of $A$, repeated by multiplicities, then the derivative at $t=1$ of the right hand side of \eqref{eq tr det t} is
\[
-\sum_{j=1}^{\dim(V)} \lambda_j \prod_{k\not=j} (1-\lambda_k).
\]
If exactly one of the eigenvalues $\lambda_j$ equals $1$, then this is the right hand side of \eqref{eq tr det}.
\end{proof}

 Recall the notation $\tau^{\nabla^F}$ for parallel transport with respect to $\nabla^F$ in Definition \ref{def length spectrum}(e).
Define the flow $\Phi$ on $ \Bigwedge T^*M \otimes F$ lifting $\varphi$ by
\[
\Phi_t (\omega \otimes v) =  (\wedge T_{\varphi_t(m)} \varphi_{-t})^* \omega \otimes \tau^{\nabla^F}_t v, 
\]
for all $t \in \R$, $m \in M$, $\omega \in \Bigwedge T^*_mM$ and $v \in F_m$.
Let $N$ be the number operator on $\Bigwedge T^*M$, acting as the scalar $p$ on $\Bigwedge^pT^*M$. As before, we assume that $\varphi$ is $g$-nondegenerate. 

Consider the smooth function $\psi_{\sigma}(t) = e^{-\sigma |t|}/|t|$ on $\Rnz$, for $\sigma \in \C$. It is not compactly supported, so a priori the pairing $ \langle \Tr_g^{\flat}((-1)^N N \Phi^*), \psi_{\sigma} \rangle$ is not defined. We define this number, if it exists, by the property that for all $\varepsilon>0$, there is an $r>0$ such that for all $\psi \in C^{\infty}_c(\Rnz)$ with values in $[0,1]$, for which $\psi(t)=1$ if $\frac{1}{r}<|t|<r$, we have
\[
\left| \langle \Tr_g^{\flat}((-1)^N N \Phi^*), \psi_{\sigma} \rangle -  \langle \Tr_g^{\flat}((-1)^N N \Phi^*), \psi \psi_{\sigma} \rangle  \right| < \varepsilon.
\]

The following fact is an equivariant generalisation of a well-known expression for the Ruelle dynamical $\zeta$-function in terms of the flat trace. See e.g.\ (5.4) in \cite{DGRS20},  (2.5) in \cite{DZ16} or (6.17)  in \cite{Shen21}.

\begin{theorem}\label{thm cont R}
Suppose that $L_{g}(\varphi)$ is countable.  Then the pairing 
\begin{equation*}
\langle \Tr_g^{\flat}((-1)^N N \Phi^*), \psi_{\sigma} \rangle
\end{equation*}
is well-defined if and only if \eqref{eq Ruelle finite L} converges, and then 
\beq{eq cont R}
R_{\varphi, \nabla^F}^g(\sigma) = \exp\left( - \langle \Tr_g^{\flat}((-1)^N N \Phi^*), \psi_{\sigma} \rangle \right).
\eeq
\end{theorem}
\begin{proof}
It follows from Lemma \ref{lem tr det A} that for all $l \in L_g(\varphi)$, $h \in G$ and $\gamma \in \Gamma_l^{hgh^{-1}}(\varphi)$, 
\beq{eq fibre trace Ruelle}
\begin{split}
\tr ( (-1)^N N hgh^{-1} \Phi_{-l}|_{\gamma}) &=
\tr \bigl( (-1)^N N  (\wedge T_{\gamma(0)} hg^{-1}h^{-1} \varphi_{l})^*\bigr ) 
\tr( (hgh^{-1} \circ \tau^{\nabla}_{-l})_{\gamma(0)})\\
&=
 -\det(1-P_{\gamma}^{hgh^{-1}}) \tr(\rho_{hgh^{-1}}(\gamma)).
 \end{split}
\eeq

Let $r>0$ be such that $\pm r \not \in L_g(\varphi)$, and $|l| > 1/|r|$ for all $l \in L_g(\varphi)$. 
Let $\psi_r \in C^{\infty}_c(\Rnz)$ be such that $\psi_r(t) = 1$  if $\frac{1}{r} < |t| < r$, and $\psi(l) = 0$ for all $l \in L_{g}(\varphi) \setminus [-r,r]$.
%

From \eqref{eq fibre trace Ruelle}, 
Theorem \ref{prop fixed pt gen} implies that
\begin{multline*}
 \langle \Tr_g^{\flat}((-1)^N N \Phi^*), \psi_r \psi_{\sigma} \rangle = \\
 -
 \int_{G/Z} \sum_{l \in L_{g}(\varphi) \cap[-r,r]}  \frac{ e^{-|l|\sigma}}{|l|} \sum_{\gamma \in \Gamma^{hgh^{-1}}_l(\varphi)} 
\sgn\bigl(\det(1-P_{\gamma}^{hgh^{-1}}) \bigr) \tr(\rho_{hgh^{-1}}(\gamma)) T_{\gamma}\, d(hZ).
\end{multline*}
The theorem follows in the limit $r \to \infty$. 
\end{proof}



\begin{remark}\label{rem Ruelle torsion det}
The equality \eqref{eq cont R} may be written suggestively as
\beq{eq Ruelle integral}
 \log R_{\varphi, \nabla^F}^g(\sigma) =- \int_{\Rnz}
 \Tr_g^{\flat}\bigl((-1)^N N \Phi_t^*\bigr)\frac{e^{-\sigma |t|}}{|t|}\, dt.
\eeq
If $\Delta_F$ has trivial kernel, then \eqref{eq def torsion} means that $2\log
T_g(\nabla^F)$ is a regularisation of the divergent integral
\beq{eq torsion integral}
  - 
\int_0^{
\infty} \Tr_g\bigl((-1)^N N e^{-t\Delta_F} \bigr) \frac{1}{t}\, dt.
\eeq
The superficial similarity between \eqref{eq Ruelle integral} and \eqref{eq torsion integral} 
is an equivariant generalisation of an analogous similarity in the classical case, compare for example (1.8) and (6.18) in \cite{Shen21}. This may be viewed as evidence 
 that an equivariant Fried conjecture \eqref{eq Fried} may hold in certain cases. 
\end{remark}


The individual numbers $T_{\gamma}$ in \eqref{eq Ruelle finite L} depend on the function $\chi$ in general, but Theorem \ref{thm cont R} implies that the expression \eqref{eq Ruelle finite L}  does not.
\begin{proposition}\label{prop R indep chi}
If the expression \eqref{eq Ruelle finite L} converges, then it does not depend on the choices 
of representatives of classes in the sets $\Gamma_l^{hgh^{-1}}(\varphi)$, the intervals $I_{\gamma}$ in \eqref{eq def Tgamma}, or 
the function $\chi$  satisfying \eqref{eq cutoff fn} for all $m$.
\end{proposition}
\begin{proof}
The right hand side of \eqref{eq cont R} does not involve choices of representatives of classes in the sets  $\Gamma_l^{hgh^{-1}}(\varphi)$, or intervals $I_{\gamma}$ as in \eqref{eq def Tgamma}. And 
Proposition 2.9 in \cite{HS25a} states that the flat $g$-trace, and hence  the  right hand side of \eqref{eq cont R}, does not depend on $\chi$.
\end{proof}

\subsection{Alternative expressions}

The integrand in \eqref{eq Ruelle finite L} is clearly invariant under right multiplication of $h$ by an element of $Z$, so that it is well-defined as a function on the quotient  $G/Z$. We reformulate \eqref{eq Ruelle finite L} in a way that makes this right $Z$-invariance property less obvious, but which can be easier to evaluate in practice. 
\begin{lemma}\label{lem Ruelle alt def}
The limit
\begin{multline}\label{eq Ruelle alt def}
\lim_{r \to \infty}
 \int_{G/Z} \sum_{l \in L_{g}(\varphi) \cap [-r,r]}  \frac{ e^{-|l|\sigma}}{|l|} \sum_{\gamma \in \Gamma^{g}_l(\varphi)} 
\sgn\bigl(\det(1-P_{\gamma}^{g}) \bigr) \tr(\rho_{g}(\gamma))
\int_{I_{\gamma}} \chi(h\gamma(s))\, ds\, 
d(hZ)
\end{multline}
converges if and only if the right hand side of \eqref{eq Ruelle finite L}  does, and then the two are equal.
\end{lemma}
\begin{proof}
Let $h \in G$. 
By $G$-equivariance of $\varphi$ and  $\nabla^F$, 
\[
\begin{split}
T_{h\gamma(0)} (\varphi_l \circ hg^{-1}h^{-1}  ) &= 
T_{\gamma(0)}h \circ
T_{\gamma(0)} (\varphi_l \circ g^{-1}  ) \circ T_{h\gamma(0)}h^{-1} \\
\rho_{hgh^{-1}}(h\gamma) &= h \circ \rho_{g}(\gamma) \circ h^{-1}.
\end{split}
\]
These equalities imply that
\[
\begin{split}
\det(1-P_{h\gamma}^{hgh^{-1}})  &= \det(1-P_{\gamma}^{g}) \\ 
 \tr(\rho_{hgh^{-1}}(h\gamma)) &= \tr(\rho_g(\gamma)), 
\end{split}
\]
respectively. Together with \eqref{eq Gamma hghinv} and the fact that $L_{hgh^{-1}}(\varphi) = L_{g}(\varphi)$, these imply the claim.
\end{proof}

\begin{remark}
It is immediate from Lemma \ref{lem Ruelle alt def} that
if the integral over $G/Z$ in \eqref{eq Ruelle alt def} may be interchanged with the sums over $L_g(\varphi)$ and $\Gamma_l^g(\varphi)$, then
\beq{eq Ruelle abs conv}
 R_{\varphi, \nabla^F}^g(\sigma) =\\
 \exp\left(
\sum_{l \in L_{g}(\varphi)}  \frac{ e^{-|l|\sigma}}{|l|} \sum_{\gamma \in \Gamma^{g}_l(\varphi)} 
\sgn\bigl(\det(1-P_{\gamma}^{g}) \bigr) \tr(\rho_{g}(\gamma)) T^g_{\gamma}
\right)
,
\eeq
where
\[
T^g_{\gamma} := 
 \int_{G/Z} \int_{I_{\gamma}} \chi(h\gamma(s))\, ds\, 
d(hZ)
\]
plays the role of a $g$-delocalised primitive period of $\gamma$. Here a Borel section of $G\to G/Z$ is used, and the total expression \eqref{eq Ruelle abs conv} does not depend on this section.
This is true in particular if \eqref{eq Ruelle finite L} converges,  $ L_{g}(\varphi)$ is finite, and $\Gamma^{g}_l(\varphi)$ is finite for all $l \in L_g(\varphi)$. 
\end{remark}

The definition of the equivariant Ruelle dynamical $\zeta$-function involving 
the limit $r \to \infty$ in \eqref{eq Ruelle finite L} is to ensure that Theorem \ref{thm cont R} holds. This then implies various other properties of the Ruelle function, such as independence of the choices made (see Proposition \ref{prop R indep chi}). In certain cases, however, a more direct definition is possible.
\begin{lemma}
Suppose that either
\begin{itemize}
\item[(a)] $L_g(\varphi)$ is finite; or
\item[(b)] $g$ is central in $G$; or
\item[(c)] $M$ and $G$ are compact; or
\item[(d)]  $G/Z$ is compact, the number
\[
\# \{\gamma \in \Gamma_l^{g}(\varphi); l \in L_g(\varphi) \cap [-r,r]\}
\]
grows at most exponentially in $r$, and there is a $C>0$ such that for all $l \in L_{g}(\varphi)$, $\gamma \in \Gamma_l^g(\varphi)$ and $h \in G$,
\[
\left| \int_{I_{\gamma}} \chi(h\gamma(s))\, ds \right| \leq C |l|.
\]
\end{itemize}
Then for all $\sigma \in \C$ for which any of these expressions converge, 
\begin{multline}\label{eq Ruelle no limit}
\log R_{\varphi, \nabla^F}^g(\sigma) = \\
\int_{G/Z} \sum_{l \in L_{g}(\varphi) }  \frac{ e^{-|l|\sigma}}{|l|} \sum_{\gamma \in \Gamma^{hgh^{-1}}_l(\varphi)} 
\sgn\bigl(\det(1-P_{\gamma}^{hgh^{-1}}) \bigr) \tr(\rho_{hgh^{-1}}(\gamma)) T_{\gamma}\, d(hZ)\\
= \int_{G/Z} \sum_{l \in L_{g}(\varphi) }  \frac{ e^{-|l|\sigma}}{|l|} \sum_{\gamma \in \Gamma^{g}_l(\varphi)} 
\sgn\bigl(\det(1-P_{\gamma}^{g}) \bigr) \tr(\rho_{g}(\gamma))
\int_{I_{\gamma}} \chi(h\gamma(s))\, ds\, 
d(hZ).
\end{multline}
\end{lemma}
\begin{proof}
Cases (a) and (b) follow directly from the definition and Lemma \ref{lem Ruelle alt def}.

If $G$ (and hence $M$) is compact, then we may take $\chi$ to be the constant function $1$. Then   the right hand sides of \eqref{eq Ruelle alt def} and \eqref{eq Ruelle no limit} both equal
\[
\vol({G/Z}) \sum_{l \in L_{g}(\varphi) }  \frac{ e^{-|l|\sigma}}{|l|} \sum_{\gamma \in \Gamma^{g}_l(\varphi)} 
\sgn\bigl(\det(1-P_{\gamma}^{g}) \bigr) \tr(\rho_{g}(\gamma))
T^{\#}_{\gamma}.
\]
So case (c) follows by (the proof of) Lemma \ref{lem Ruelle alt def}.

In case (d), the right hand side of \eqref{eq Ruelle no limit} converges absolutely, and the claim follows by Lemma \ref{lem Ruelle alt def} and dominated convergence.
\end{proof}


\subsection{Restriction to subgroups}

The equivariant Ruelle $\zeta$-function behaves naturally with respect to restriction to cocompact subgroups. This is analogous to the corresponding property of equivariant analytic torsion, Proposition 2.4 in \cite{HS22a}. 
\begin{lemma}\label{lem Ruelle subgroup}
Let $H<G$ be a cocompact, unimodular, closed subgroup containing $g$. Let $Z_H:= Z \cap H$ be the centraliser of  $g$ in $H$, and suppose that $H/Z_H$ has an $H$-invariant measure. 
%
The expression \eqref{eq Ruelle finite L}  for fixed $r$ converges absolutely if and only if the same expression for the restricted action by $H$ on $M$ converges absolutely. If this is true for all $r>0$, then the equivariant Ruelle $\zeta$-function $R^{g, H}_{\varphi, \nabla^F}$ for the action by $H$  is related to  the equivariant Ruelle $\zeta$-function for the action by $G$ by
\beq{eq Ruelle subgroup}
R^{g, H}_{\varphi, \nabla^F} = 
(R^g_{\varphi, \nabla^F})^{\vol(Z/Z_H)}.
\eeq
Here the volume $\vol(Z/Z_H)$ is computed with respect to the measure compatible with the Haar measures on $G$ and $H$ and the given measures on $G/Z$ and $H/Z_H$. 

If $G$ is compact, then the same statement is true with absolute convergence replaced by convergence.
%
\end{lemma}
\begin{proof}
Let $\psi \in C^{\infty}_c(G)$ be such that for all $x \in G$, 
\[
\int_H \psi(xh^{-1})\, dh = 1.
\]
Define $\chi_H \in C^{\infty}_c(M)$ by
\[
\chi_H(m) := \int_G \psi(x) \chi(xm)\, dx
\]
for $m \in M$. 
Then  \eqref{eq cutoff fn} holds  for all $m \in M$, with $\chi$ replaced by $\chi_H$ and $G$ replaced by $H$. 

Let $r>0$.
By  
right $Z$-invariance of the integrand, 
\begin{multline*}
 \int_{G/Z} \sum_{l \in L_{g}(\varphi) \cap [-r,r]}  \frac{ e^{-|l|\sigma}}{|l|} \sum_{\gamma \in \Gamma^{g}_l(\varphi)} 
\sgn\bigl(\det(1-P_{\gamma}^{g}) \bigr) \tr(\rho_{g}(\gamma))
\int_{I_{\gamma}} \chi(h\gamma(s))\, ds\, 
d(hZ) \\
=
\frac{1}{\vol(Z/Z_H)}
 \int_{G/{Z_H}} \sum_{l \in L_{g}(\varphi) \cap [-r,r]}  \frac{ e^{-|l|\sigma}}{|l|} \sum_{\gamma \in \Gamma^{g}_l(\varphi)} 
\sgn\bigl(\det(1-P_{\gamma}^{g}) \bigr) \tr(\rho_{g}(\gamma))
\int_{I_{\gamma}} \chi(x\gamma(s))\, ds\, 
d(xZ_H).
\end{multline*}
This expression converges absolutely if and only if the following expression does, and then the two are equal:
\beq{eq Ruelle subgroup 2}
\frac{1}{\vol(Z/Z_H)}
\sum_{l \in L_{g}(\varphi) \cap [-r,r]}  \frac{ e^{-|l|\sigma}}{|l|} \sum_{\gamma \in \Gamma^{g}_l(\varphi)} 
\sgn\bigl(\det(1-P_{\gamma}^{g}) \bigr) \tr(\rho_{g}(\gamma))
 \int_{G/Z_H} 
\int_{I_{\gamma}} \chi(x\gamma(s))\, ds\, 
d(xZ_H).
\eeq
For all $l \in L_{g}(\varphi)$ and $\gamma \in \Gamma^{g}_l(\varphi)$,
\beq{eq prim per G H}
 \int_{G/Z_H} 
\int_{I_{\gamma}} \chi(x\gamma(s))\, ds\, 
d(xZ_H) =  
\int_{H/Z_H} 
\int_{I_{\gamma}} \chi_H(h\gamma(s))\, ds\, 
d(hZ_H).
\eeq
The maps $P_{\gamma}^g$ and $\rho_g(\gamma)$ are the same for the actions by $G$ and $H$. So by Lemma \ref{lem Ruelle alt def}, 
\begin{multline}\label{eq Ruelle alt def 2}
\log R_{\varphi, \nabla^F}^{g, H}(\sigma) =\\
\lim_{r \to \infty}
 \int_{H/Z_H} \sum_{l \in L_{g}(\varphi) \cap [-r,r]}  \frac{ e^{-|l|\sigma}}{|l|} \sum_{\gamma \in \Gamma^{g}_l(\varphi)} 
\sgn\bigl(\det(1-P_{\gamma}^{g}) \bigr) \tr(\rho_{g}(\gamma))
\int_{I_{\gamma}} \chi_H(h\gamma(s))\, ds\, 
d(hZ_H).
\end{multline}
By \eqref{eq prim per G H}, the right hand side of \eqref{eq Ruelle alt def 2} converges absolutely if and only if \eqref{eq Ruelle subgroup 2} does, and then  $\eqref{eq Ruelle subgroup 2}$ equals $1/\vol(Z/Z_H)$ times \eqref{eq Ruelle alt def 2}. This implies the claim for possibly noncompact $G$.

If $G$, and hence $M$ and $H$, is compact, then we may use constant cutoff functions $\chi = 1/\vol(G)$ and $\chi_H = 1/\vol(H)$ (these volumes may be normalised to $1$). Then Lemma \ref{lem Ruelle alt def} and Example 2 below Remark \ref{rem P rho dep l} imply that, whenever the following expressions converge, 
\[
\begin{split}
\log R^g_{\varphi, \nabla^F}(\sigma) &= 
\frac{ \vol(G/Z)}{\vol(G)} \sum_{l \in L_{g}(\varphi)}  \frac{ e^{-|l|\sigma}}{|l|} \sum_{\gamma \in \Gamma^{g}_l(\varphi)} 
\sgn\bigl(\det(1-P_{\gamma}^{g}) \bigr) \tr(\rho_{g}(\gamma))
T_{\gamma}^{\#}\\
 &=    \frac{ \vol(G/Z)}{\vol(G)} \frac{\vol(H)}{ \vol(H/Z_H)}  \log R^{g, H}_{\varphi, \nabla^F}(\sigma)\\
 &= \frac{1}{\vol(Z/Z_H)} \log R^{g, H}_{\varphi, \nabla^F}(\sigma),
\end{split}
\]
which implies the claim.
\end{proof}

\begin{remark}
If $G$ and $M$ are compact, $g=e$, and $\varphi$ is Anosov, then the last claim in Lemma \ref{lem Ruelle subgroup} and the exponential bound on the number of flow curves with periods less than a given value  \cite{Margulis} imply that \eqref{eq Ruelle finite L} converges for $\sigma$ with large real part.
\end{remark}

\begin{remark}
If $G$ is noncompact, then absolute convergence is important in Lemma \ref{lem Ruelle subgroup}. See Remark \ref{rem conv Rn} for an example where convergence is not absolute and the conclusion of  Lemma \ref{lem Ruelle subgroup} does not hold.
\end{remark}

In the setting of  
Lemma \ref{lem Ruelle subgroup}, the equivariant Ruelle $\zeta$-function has the same behaviour with respect to restriction to cocompact subgroups as equivariant analytic torsion, see Proposition 2.4 in \cite{HS22a}. This has the following immediate consequence. 
\begin{proposition}
In the setting of Lemma \ref{lem Ruelle subgroup}, the answer to Question \ref{q Fried} is \emph{yes} for the action by $G$ on $M$ if and only if the answer is \emph{yes} for the action by the cocompact subgroup $H$. 
\end{proposition}

\begin{example}
If $G$ is compact, then Question \ref{q Fried} is equivalent to the same question for the action by the closure of the subgroup generated by $g$.
\end{example}

\begin{example}\label{ex Fried loc symm}
Suppose that $G$ is a Lie group, $K<G$ is maximal compact, and $\Gamma<G$ is discrete and cocompact, and contains $g$. Suppose that $M = G/K$. 
If the equivariant Ruelle $\zeta$-function converges absolutely for large enough $\sigma$, then
 Question \ref{q Fried} 
 for the action by $G$ is equivalent to the same question for the action by $\Gamma$. In this way, the equivariant Fried conjecture for the action by $G$ on $G/K$ contains information about all compact locally symmetric spaces $\Gamma \backslash G/K$, for $\Gamma<G$ as in this example. 
\end{example}

\subsection{Relations with the classical Ruelle dynamical $\zeta$-function}


Suppose for now that $G = \{e\}$ is trivial and $M$ is compact and orientable. (We will drop this assumption again in the text leading up to Proposition \ref{prop R XM}.) Now  $\chi \equiv 1$. 
For $l \in L_e(\varphi)$, we write
\[
H(l) :=
\sum_{\gamma \in \Gamma^{e}_l(\varphi)} 
\sgn\bigl(\det(1-P_{\gamma}^{e}) \bigr) \tr(\rho_e(\gamma)) T_{\gamma}.
\]

We fix a Hermitian metric on $F$, and say that $\rho_e$ is unitary if  $\rho_e(\gamma)$ is unitary for all $\gamma \in \Gamma^{e}_l(\varphi)$.
\begin{lemma}\label{lem H}
Let $l \in L_e(\varphi)$. Then $-l \in L_{e}(\varphi)$.
If   $\varphi$ is orientation-preserving  and $\rho_e$ is unitary,  then
\[
H( -l) = (-1)^{\dim(M)-1} \overline{H(l)}.
\]
\end{lemma}
\begin{proof}
The first claim follows from the fact that for all $m \in M$ and $l \in \R$, we have $\varphi_l(m) = m$ if and only if $\varphi_{-l}(m) = m$.
 This also shows that for all $l \in L_e(\varphi)$, the sets $\Gamma_l^e(\varphi)$ and $\Gamma_{-l}^e(\varphi)$ are equal.

Recall that the maps $P_{\gamma}^e$ and $\rho_e(\gamma)$ depend on $l$, if $\gamma \in \Gamma_{l}^e(\varphi)$. Fix $l \in L_e(\varphi)$ and $\gamma \in \Gamma_l^e(\varphi) =  \Gamma_{-l}^e(\varphi) $. We write $(P_{\gamma}^e)_{\pm l}$ for the map $P_{\gamma}^e$, and $\rho_e(\gamma)_{\pm l}$ for the map $\rho_e(\gamma)$,  associated to $\gamma$ when viewed as an element of  $\Gamma_{\pm l}^e(\varphi)$. Then
\[
\begin{split}
(P_{\gamma}^e)_{- l}&= (P_{\gamma}^e)_{ l}^{-1};\\
\rho_e(\gamma)_{- l} &= \rho_e(\gamma)_{ l}^{-1}.
\end{split}
\]
If $\rho_e$ is unitary, then the second equality implies that
\beq{eq rho conj}
\tr(\rho_{e}(\gamma)_{-l}) = \overline{\tr(\rho_e(\gamma)_l)}.
\eeq

Now for any invertible endomorphism $A$ of a finite-dimensional vector space $V$,
\[
\sgn \det(I-A^{-1}) = (-1)^{\dim(V)} \sgn(\det(A)) \sgn \det(I-A).
\]
We apply this to $A = (P_{\gamma}^e)_{ l}$, 
and note that this map has positive determinant because $T\varphi_l$  preserves the orientation on the quotient of $TM$ by the span of the image of $u$. This yields
\[
\sgn \det(I-(P^{e}_{\gamma})_{-l}) = (-1)^{\dim(M)-1}\sgn \det(I-(P^{e}_{\gamma})_l). 
\]
The claim follows from this equality, \eqref{eq rho conj}, and the fact that $T_{\gamma}$ does not depend on $l$. 
\end{proof}

We write $L^+(\varphi) := L_e(\varphi) \cap (0,\infty)$. The classical Ruelle dynamical $\zeta$-function $R_{\varphi, \nabla^F}$ for Anosov flows is defined by
\[
  R_{\varphi, \nabla^F}(\sigma):=\exp \left(\sum_{l \in L^+(\varphi)} \frac{e^{l\sigma}}{l}H( l) \right),
\]
 for $\sigma \in \C$ for which the expression converges.  (See also Example 2 below Remark \ref{rem P rho dep l}.)
\begin{lemma}\label{lem Rh cpt}
If $\varphi$ is orientation-preserving  and $\rho_e$ is unitary, then on  the real line,
\beq{eq Rh cpt}
R^{e}_{\varphi, \nabla^F}= 
\left\{
\begin{array}{ll}
|R_{\varphi, \nabla^F}|^2 & \text{if $\dim(M)$ is odd};\\
\left( \frac{R_{\varphi, \nabla^F}}{|R_{\varphi, \nabla^F}|}\right)^2 & \text{if $\dim(M)$ is even}.
\end{array}
\right.
\eeq
This equality means that the left hand side converges  at $\sigma \in \R$ if and only if the right hand side does, and then  the two are equal.
\end{lemma}
\begin{proof}
Let $\sigma \in \C$. 
By  Lemma \ref{lem H}, 
\beq{eq R tilde R}
\log
 R_{\varphi, \nabla^F}^e(\sigma)  =\sum_{l \in L^+(\varphi)}  \frac{e^{-l\sigma}}{l} \left(H( l)+  (-1)^{\dim(M)-1}\overline{ H(l) }\right).
\eeq
Suppose that $\sigma$ is real. 
If $\dim(M)$ is odd,  then \eqref{eq R tilde R} becomes
\beq{eq R tilde R odd}
\log
 R_{\varphi, \nabla^F}^e(\sigma) = 2\Real \bigl(\log   R_{\varphi, \nabla^F}(\sigma) \bigr).
\eeq
If $\dim(M)$ is even, then \eqref{eq R tilde R} becomes
\beq{eq R tilde R even}
\log
 R_{\varphi, \nabla^F}^e(\sigma) = 2i\Imag \bigl(\log   R_{\varphi, \nabla^F}(\sigma) \bigr).
\eeq
 The equalities \eqref{eq R tilde R odd} and \eqref{eq R tilde R even} can be rewritten as \eqref{eq Rh cpt}. 
\end{proof}



There is a more interesting and useful relation between the equivariant Ruelle dynamical $\zeta$-function and its classical counterpart than  Lemma \ref{lem Rh cpt}.
Suppose that $X$ is a compact  manifold, and let $M$ be the universal cover of $X$. Let
 $F_X = M \times_{\rho} \C^r \to X$ be the  flat vector bundle associated to a representation $\rho\colon \pi_1(X) \to \GL(r, \C)$.
 Consider the trivial connection $\nabla^F = d$ on $F = M \times \C^r$.


Let $\varphi_X$ be a flow on $X$, and $\varphi$ its lift to a $\Gamma$-equivariant flow on $M$. 
Let $R_{\varphi_X, \rho}$ be the classical Ruelle dynamical $\zeta$-function associated to $\varphi_X$ and $\rho$.
\begin{proposition}\label{prop R XM}
On the real line, 
\beq{eq R XM}
 \prod_{(g)}  R^{g}_{\varphi, \nabla^F}=
 \left\{
\begin{array}{ll}
|R_{\varphi_X, \rho}|^2 & \text{if $\dim(M)$ is odd};\\
\left( \frac{R_{\varphi_X, \rho}}{|R_{\varphi_X, \rho}(\sigma)|}\right)^2 & \text{if $\dim(M)$ is even}.
\end{array}
\right.
\eeq
Here $\prod_{(g)}$ is the product over all conjugacy classes $(g) \subset \pi_1(X)$, and the equality means that the left hand side converges  at $\sigma \in \R$ if and only if the right hand side does, and then  the two are equal.
\end{proposition}
\begin{proof}
Let $N_X$ be the number operator on $\Bigwedge^*T^*X$.
Proposition 2.16 in \cite{HS25a}  implies that
\beq{eq flat trace decomp}
\Tr^{\flat}((-1)^{N_X}N_X \Phi_X^*) = 
\sum_{(\gamma)}
\Tr^{\flat}_{\gamma}((-1)^{N}N \Phi^*).
\eeq
Let $\nabla^{F_X}$ be the connection on $F_X$ induced by $\nabla^F = d$. Then
by Theorem \ref{thm cont R} and \eqref{eq flat trace decomp}, 
\[
R_{\varphi_X, \nabla^{F_X}}^e= \prod_{(\gamma)}R_{\varphi, \nabla^{F}}^{\gamma}.
\]
Hence the claim follows from Lemma \ref{lem Rh cpt}.
\end{proof}

\subsection{The classical decomposition of the Ruelle dynamical $\zeta$-function for geodesic flow}\label{sec class decomp}

The classical Ruelle dynamical $\zeta$-function for geodesic flows on the sphere bundle of a compact, oriented, negatively curved Riemannian manifold $X$ can be decomposed into contributions from different conjugacy classes of the fundamental group of $X$, analogously to Proposition \ref{prop R XM}. These contributions are equal to the relevant equivariant Ruelle dynamical $\zeta$-functions, as we now discuss.

Let $X$ be a compact, oriented Riemannian manifold with negative sectional curvature. Let $\varphi_X$ be the geodesic flow on the unit sphere bundle $S(TX)$. Let $\tilde X$ be the universal cover of $X$. Let $\varphi$ be the lift of $\varphi_X$ to $M = S(T\tilde X)$; this is the geodesic flow on the universal cover. Consider the natural action by $G = \pi_1(X)$ on $M$. Let $g \in G$ be nontrivial. From the assumption that $X$ has negative sectional curvature, it follows that there is a unique geodesic $\gamma$ in $X$ in the conjugacy class of $g$. 
%
Let $l$ be the period of $\gamma$, and let
\[
T^{\#}_{\gamma} = \min\{t>0; \gamma(t) = \gamma(0)\}
\]
 be the primitive period of $\gamma$. 

Let $\rho\colon G \to \U(r)$ be a unitary representation. Consider the trivial vector bundle $M \times \C^r \to M$, on which $G$ acts by
\[
x\cdot (m, v) = (xm, \rho(x)v),
\]
for $x \in G$, $m \in M$ and $v \in \C^r$. Let $\nabla^F = d \otimes 1_{\C^r}$ be the natural flat connection on $F$.
\begin{lemma}\label{lem geodesic}
For all $\sigma  \in \C$, 
\beq{eq R geodesic}
 R^g_{\varphi, \nabla^F}(\sigma) = \exp\left( \frac{e^{-l\sigma}}{l}
\tr(\rho(g))T^{\#}_{\gamma} \right).
\eeq
\end{lemma}
\begin{proof}
In this case  $L_g(\varphi) = \{l\}$, so \eqref{eq Ruelle no limit} holds. 
Then
\[
\bigcup_{h \in G} \Gamma_l^{hgh^{-1}}(\varphi) = \{\gamma\}.
\]

As $X$ is oriented, 
\beq{eq P geodesic}
\sgn(\det(1-P_{\gamma}^{hgh^{-1}}))  = 1
\eeq
for all $h \in G$. Because 
parallel transport with respect to $\nabla^F$ is trivial, we have
\[
\rho_{hgh^{-1}}(\gamma) = \rho(hgh^{-1})
\]
for all $h$, so 
\beq{eq rho geodesic}
\tr (\rho_{hgh^{-1}}(\gamma) ) = \tr(\rho(g)).
\eeq

The primitive period $T_{\gamma}$  a priori may depend on the choice of the cutoff function $\chi$. As the left hand side of \eqref{eq R geodesic} does not depend on  $\chi$, by Proposition \ref{prop R indep chi}, we are free to choose this function in a suitable way. 
We choose $\chi$ to be an approximation of the indicator function of a suitable fundamental domain $Y \subset M$, in the following way.

The subset $\gamma([0, T^{\#}_g)) \subset M$ has the property that for all $m \in \gamma([0, T^{\#}_g)) $ and $x \in G \setminus \{e\}$, the point $xm$ lies outside $\gamma([0, T^{\#}_g))$. Therefore, the set $\gamma([0, T^{\#}_g)) $ is contained in a fundamental domain $Y$ for the action by $G$ on $M$. We choose such a fundamental domain $Y$ such that
\begin{enumerate}
\item $\gamma(\R) \cap Y = \gamma([0, T^{\#}_g))$; and
\item the closure of $Y$ equals the closure of the interior of $Y$.
\end{enumerate}
By the second property, the indicator function $1_Y$ can be approximated by compactly supported smooth functions arbitrarily close in the $L^1$-norm. As $\chi \in C^{\infty}_c(M)$ approaches $1_Y$ in the $L^1$-norm, 
\beq{eq Tgam geodesic}
T_{\gamma} = 
\int_{\R} \chi(\gamma(s))\, ds \to \int_{\R}  1_Y(\gamma(s))\, ds = T^{\#}_{\gamma}.
\eeq

By \eqref{eq P geodesic}, \eqref{eq rho geodesic} and \eqref{eq Tgam geodesic},  the left hand side of \eqref{eq R geodesic} equals
 the right  hand side of \eqref{eq R geodesic}. 
\end{proof}

\begin{remark}
Lemma \ref{lem geodesic} means that 
$ R^g_{\varphi, \nabla^F}$ is the contribution to the classical Ruelle dynamical $\zeta$-function associated to the conjugacy class of $g$, see for example Definition 5.4 in \cite{Shen18}. This can be used to prove Proposition \ref{prop R XM} in the case of geodesic flow. 
\end{remark}

\section{The line and the circle}\label{sec ZR}

In the rest of this paper, we compute the equivariant Ruelle dynamical $\zeta$-function in several classes of examples. We will see that this function is defined for some flows for which the classical Ruelle $\zeta$-function is not defined. We will also see  that the equivariant Fried equality \eqref{eq Fried} holds in the examples where both sides are defined.

We start with the most basic, one-dimensional examples: the line and the circle.
We will then show that the equivariant Fried conjecture is true, by explicit computations.

\subsection{The real line}\label{sec R}

Consider the action of $G = \R$ on $M = \R$ by addition, and the flow
\[
\varphi_t(x) = x+t
\]
on $\R$. Consider the trivial line bundle $F = \R \times \C \to \R$. Let $\alpha \in i\R$, and consider  the connection $\nabla^F := d+\alpha dx$ on $F$. We let $\R$ act trivially on the fibres of $F$.
\begin{remark}\label{rem equiv R}
One obtains the same result for the trivial connection $d$ and the nontrivial action by $\R$ on $F$ given by
\[
g \cdot (x,z) = (x+g, e^{\alpha g}z),
\]
for $g,x \in \R$ and $z \in \C$. See Remark 6.9 in \cite{HS22a}.
\end{remark}

\begin{lemma}\label{lem Ruelle R}
If $g \in \Rnz$, then for all $\sigma \in \C$,
\beq{eq Ruelle R g}
R^g_{\varphi, \nabla^F} = \exp\left(\frac{e^{\alpha g -|g|\sigma}}{|g|} \right),
\eeq
and
\beq{eq Ruelle R 0}
R^0_{\varphi, \nabla^F} = 1.
\eeq
\end{lemma}
\begin{proof}
Let $\gamma \colon \R \to \R$ be the unique flow curve of $\varphi$ modulo time shifts:
$
\gamma(t) = t.
$
Suppose first that $g \in \Rnz$. Then $L_g(\varphi) = \{g\}$, and $\Gamma^g_g(\varphi) = \{\gamma\}$. So \eqref{eq Ruelle no limit} implies that
\[
R^g_{\varphi, \nabla^F} = \exp\left(
 \frac{ e^{-|g|\sigma}}{|g|} 
\sgn\bigl(\det(1-P_{\gamma}^{g}) \bigr) \tr(\rho_{g}(\gamma)) T_{\gamma}
\right).
\]

Now $u(x)^{\perp} = \{0\}$ for all $x \in \R$, so the flow is nondegenerate, and 
\[
\sgn (\det(1-P^g_{\gamma})) = 1.
\]
For all $a<b$, 
parallel transport with respect to $\nabla^F$ along $\gamma|_{[a,b]}$ is given by
\[
z \mapsto e^{(b-a)\alpha}z,
\]
for $z \in \C = F_{a}$. So
\[
\rho_g(\gamma) z = e^{\alpha g}.
\]
Now $T_{\gamma} = 1$ by Example 2 below Remark \ref{rem P rho dep l}, so
 \eqref{eq Ruelle R g} follows.

For $g=0$, we have $L_0(\varphi) = \emptyset$, which implies \eqref{eq Ruelle R 0}.
\end{proof}

Proposition 6.2 in \cite{HS22a} states that the equivariant analytic torsion of $\nabla^F$ equals
\beq{eq torsion R}
 T_g(\nabla^F) = \exp\left( \frac{e^{\alpha g}}{2 |g| } \right),
\eeq
for $g \in \Rnz$, while
\[
T_0(\nabla^F) = 1.
\]
It follows by Lemma \ref{lem Ruelle R} that the answer to Question \ref{q Fried} is \emph{yes} in this example, and the equivariant Fried conjecture holds.

\begin{remark} \label{rem neg l R}
The two sides of \eqref{eq Fried} are generally not real in this example. The same is true in other cases; see Lemma \ref{lem Ruelle S1} and Proposition \ref{prop Ruelle Eucl}. This shows that it is desirable not to have absolute values in \eqref{eq Fried}. 

A version of the equivariant Ruelle $\zeta$-function involving only positive $l$ would be equal to $1$ for negative $g$ in this example, see Example 2 below Remark \ref{rem P rho dep l}. Because equivariant analytic torsion \eqref{eq torsion R} is different from $1$ for such $g$, the equivariant Fried conjecture would not hold for such a definition of the equivariant Ruelle $\zeta$-function. The two sides of \eqref{eq Fried} do not even have the same absolute value in that case.
\end{remark}

\subsection{The circle}\label{sec ex circle}

The example in this subsection generalises Example 2.10  in \cite{Shen21}.

We will use the following fact from complex analysis.
\begin{lemma}\label{lem mer cont gen}
Let $r \in \R \setminus \Z$, and $\alpha \in i\R$. For $z \in \C$ with positive real part, define
\[
F(z) := \sum_{n \in \Z} \frac{e^{\alpha (n+r) - |n+r|z}}{|n+r|}.
\]
Then for all $z \in \C \setminus (\pm \alpha + 2\pi i \Z)$, there is a neighbourhood of $z$ to which $F$ continues analytically. In particular, if $\alpha \not \in 2\pi i \Z$, then $F$ has an analytic continuation to a neighbourhood of zero.
\end{lemma}
\begin{proof}
We assume without loss of generality that $r \in (0,1)$.
If $\Real(z)>0$, then
\beq{eq Fz}
F(z) = \sum_{n =0}^{\infty} \frac{e^{(n+r)(\alpha - z)}}{n+r} - 
 \sum_{n =0}^{\infty} \frac{e^{(-n+r)(\alpha + z)}}{-n+r} + \frac{e^{r(\alpha+z)}}{r}.
\eeq
The last term is analytic, so it is enough to show that the two series on the right hand side have the desired properties. These series can be written in terms of the hypergeometric function ${_2 F_1}$. Using the Pochhammer symbol $(w)_n = w \cdot (w+1) \cdots (w+n-1)$, and $(w)_0 = 1$,  we have for all integers $n \geq 0$,
\begin{align}
\frac{1}{n+r} &= \frac{1}{r} \frac{(1)_n (r)_n}{(r+1)_n} \frac{1}{n !}; \label{eq Poch 1}\\
\frac{1}{-n+r} &= \frac{1}{r} \frac{(1)_n (-r)_n}{(-r+1)_n} \frac{1}{n !}.\label{eq Poch 2}
\end{align}
The equality \eqref{eq Poch 1} implies that the first series on the right hand side of \eqref{eq Fz} equals
\[
\frac{e^{r(\alpha - z)}}{r}
\sum_{n =0}^{\infty}
\frac{(1)_n (r)_n}{(r+1)_n} \frac{\bigl(e^{\alpha - z}\bigr)^n }{n !} = \frac{e^{r(\alpha - z)}}{r} {}_2 F_1(1, r; r+1; e^{\alpha - z}). 
\]
Similarly, \eqref{eq Poch 2} implies that the second series on the right hand side of \eqref{eq Fz} equals
\[
\frac{e^{r(\alpha + z)}}{r} {}_2 F_1(1, -r; -r+1; e^{-\alpha - z}).
\]

The function ${}_2 F_1(a, b; c; \relbar)$ has an analytic continuation to a neighbourhood of any $z \in \C \setminus \{1\}$; this is a standard fact, see e.g.\ Section 5.2 of \cite{Temme96}. This implies the claim. We point out that
the condition $\Real(c-a-b)>0$ for the function ${}_2 F_1(a, b; c; \relbar)$ to extend to $z=1$ does not apply in our situation, as $c-a-b = 0$ for both combinations of $a$, $b$ and $c$ that occur.
\end{proof}

We consider the action by the circle $\R/\Z$ on itself by left multiplication, and the flow $\varphi$ given by
\[
\varphi_t(x+\Z) = x+t+\Z
\]
for $x,t \in \R$.
 We consider the connection $\nabla^F = d+\alpha dx$ on the trivial line bundle  $F \to \R/\Z$, with the trivial action by $\R/\Z$, for $\alpha \in i\R$.
\begin{remark}Analogously to Remark \ref{rem equiv R}, one can equivalently consider the connection $d$ on the line bundle associated to the unitary representation of $\Z = \pi_1(\R/\Z)$ defined by $\alpha$. 
\end{remark}

\begin{lemma}\label{lem Ruelle S1}
If  $g = r+\Z \in \R/\Z$ for $r \not \in \Z$, then for all $\sigma \in \C$ with positive real part, 
\beq{eq Ruelle S1}
R^g_{\varphi, \nabla^F}(\sigma) = \exp \left( 
\sum_{n \in \Z} \frac{e^{\alpha(n+r)-|n+r|\sigma}}{|n+r|}
\right).
\eeq
If $\alpha \not \in 2\pi i \Z$, then this has an analytic continuation to a neighbourhood of $\sigma = 0$.
\end{lemma}
\begin{proof}
Let $\gamma$ be the unique flow curve of $\varphi$, modulo time shifts: $\gamma(t) = t+\Z$. Let $r \not \in \Z$ and let $g = r+\Z$. Then $L_g(\varphi) = \{n+r; n \in \Z\}$. For all $n \in \Z$, we have $\Gamma_{n+r}^g(\varphi) = \{\gamma\}$. 
Now $u(x+\Z)^{\perp} = \{0\}$ for all $x \in \R$, so the flow is nondegerenate, and 
\[
\sgn (\det(1-P^g_{\gamma})) = 1.
\]

For all $a<b$, 
parallel transport with respect to $\nabla^F$ along $\gamma|_{[a,b]}$ is given by
\[
z \mapsto e^{(b-a)\alpha}z,
\]
for $z \in \Z = F_{a+\Z}$. For $|a-b|<1$, this follows from the solution of the equation $s'+\alpha s = 0$ for parallel sections; in general this follows by composing parallel transport over intervals of lengths smaller than $1$.
So if $\gamma$ is viewed as an element of $\Gamma_{n+r}^g(\varphi)$, then 
\[
\rho_g(\gamma) = e^{\alpha(n+r)}.
\]
Now $T_{\gamma} = 1$  by Example 2 below Remark \ref{rem P rho dep l}. 
Since $G$ is abelian, \eqref{eq Ruelle S1} holds
whenever the right hand side converges.
The second claim follows by Lemma \ref{lem mer cont gen}. 
\end{proof}

\begin{example}\label{ex r=1/2}
Using the Taylor series of the function $x \mapsto \frac{\tanh^{-1}(x)}{\sqrt{x}}$, one finds that for  $z \in \C$ with negative real part, 
\[
\sum_{n=0}^{\infty} \frac{e^{(n+\frac{1}{2})z}}{n+\frac{1}{2}} = 2\tanh^{-1}(e^z).
\]
Together with the equality $e^{2w} = \frac{1+\tanh(w)}{1-\tanh(w)}$, this implies that
\[
\exp\left( \sum_{n=0}^{\infty} \frac{e^{(n+\frac{1}{2})z}}{n+\frac{1}{2}}  \right) = 
 \frac{1+e^z}{1-e^z}.
\]
So as a special case of Lemma \ref{lem Ruelle S1}, we find that for $\sigma \in \C$ with positive real part, 
\[
\begin{split}
R^{\frac{1}{2}+\Z}_{\varphi, \nabla^F}(\sigma) &=\exp\left( 
\sum_{n=0}^{\infty}  \frac{e^{(n+\frac{1}{2})(\alpha - \sigma)}}{n+\frac{1}{2}} + 
\sum_{n=0}^{\infty}  \frac{e^{(n+\frac{1}{2})(-\alpha - \sigma)}}{n+\frac{1}{2}}\right)
\\
&= 
 \frac{(1+e^{\alpha - \sigma})(1+e^{-\alpha - \sigma})}{(1-e^{\alpha - \sigma})(1-e^{-\alpha - \sigma})}.
\end{split}
\]
If $\sigma$ is a positive real number, then this equals $\left|\frac{1+e^{\alpha - \sigma}}{1-e^{\alpha - \sigma}} \right|^2= |\tanh\left(\frac{\alpha- \sigma}{2} \right)|^{-2}$.
\end{example}

Suppose that $\alpha \not \in 2\pi i \Z$; this is equivalent to the condition that $\ker(\Delta_F^p) = \{0\}$ in Question \ref{q Fried}.
Then by
Proposition 6.6 in \cite{HS22a}, we have  for $g = r+ \Z$, with $r \not \in \Z$, 
\beq{eq torsion S1}
T_g(\nabla^F) = \exp \left(\frac{1}{2}  \sum_{n \in \Z}  \frac{{1}}{|n-r|} e^{- \alpha (n-r)} \right).
\eeq
By Lemma \ref{lem Ruelle S1} and \eqref{eq torsion S1}, and the substitution $n \mapsto -n$,
\[
R_{\varphi, \nabla^F}^g(0) = T_g(\nabla^F)^2.
\]
So the answer to Question \ref{q Fried} is \emph{yes} in this case.

The following computation is analogous to Example 2.10 in \cite{Shen21}.
\begin{lemma}\label{lem Ruelle S1 e}
Suppose that $\alpha \not \in 2\pi i \Z$. Then
for $e=0+\Z \in G$ and all $\sigma \in \C$ with positive real part, 
\[
R^{e}_{\varphi, \nabla^F}(\sigma) = \bigl( (1-e^{\alpha - \sigma})(1-e^{-\alpha - \sigma}) \bigr)^{-1}.
\]
\end{lemma}
\begin{proof}
Now $L_{e}(\varphi) = \Z \setminus \{0\}$. For all $n \in \Z \setminus \{0\}$, we have $\Gamma_n^e(\varphi) = \{\gamma\}$, where $\gamma$ is the unique flow curve as in the proof of Lemma \ref{lem Ruelle R}. Analogously to the proof of that lemma, we find that whenever the following converges, 
\beq{eq Ruelle S1 e 1}
\log R^e_{\varphi, \nabla^F}(\sigma) 
= \sum_{n \in \Z \setminus \{0\}} \frac{e^{-|n|\sigma}}{|n|} e^{n\alpha} 
= \sum_{n=1}^{\infty}\frac{e^{n(\alpha - \sigma)}}{n} + \sum_{n=1}^{\infty}\frac{e^{-n(\alpha + \sigma)}}{n}. 
\eeq
Using the Taylor series of the function $x\mapsto \log(1+x)$, we find that if $\Real(\sigma)>0$,
\[
\log R^e_{\varphi, \nabla^F}(\sigma) = -\log(1-e^{\alpha - \sigma}) - \log(1-e^{-\alpha - \sigma}),
\]
which implies the claim.
\end{proof}


If $\alpha \not \in 2\pi i \Z$, then by Example 1.7 in \cite{Shen21}
(see also Proposition 6.6 in \cite{HS22a}), 
we have for 
 $e = 0+\Z$, 
\beq{eq torsion S1 e}
T_e(\nabla^F)=
 \left( - (2\sinh(\alpha/2)^2)\right)^{-1/2}. 
\eeq
By Lemma \ref{lem Ruelle S1 e} and \eqref{eq torsion S1 e}, 
\[
R^e_{\varphi, \nabla^F}(0) = T_e(\nabla^F)^2.
\]
So also in this case, the answer to Question \ref{q Fried} is \emph{yes}.

Now consider the situation of Lemma \ref{lem Ruelle R}, with  
the group $\R$ replaced by $\Z$. Then by  arguments entirely analogous to the proof of that lemma, we find that
 for all $g \in \Z \setminus \{0\}$ and $\sigma \in \C$,
\beq{eq Ruelle Z g}
R^g_{\varphi, \nabla^F} = \exp\left(\frac{e^{\alpha g -|g|\sigma}}{|g|} \right)
\eeq
and that
\beq{eq Ruelle Z 0}
R^0_{\varphi, \nabla^F} = 1.
\eeq
Together with Lemma \ref{lem Ruelle S1 e}, this allows us to illustrate  Proposition \ref{prop R XM}. 

By \eqref{eq Ruelle Z g} and \eqref{eq Ruelle Z 0}, we have for all $\sigma \in \C$ for which the following converges, 
\[
\sum_{g \in \Z}  \log R^g_{\varphi, \nabla^F}(\sigma) = \sum_{n=1}^{\infty}\frac{e^{n(\alpha - \sigma)}}{n} + \sum_{n=1}^{\infty}\frac{e^{-n(\alpha + \sigma)}}{n}.
\]
By \eqref{eq Ruelle S1 e 1}, the right hand side equals $\log R^{0+\Z}_{\varphi_{\R/\Z}, \nabla^{F_{\R/\Z}}}(\sigma)$, where $\varphi_{\R/\Z}$ and $\nabla^{F_{\R/\Z}}$ are induced by $\varphi$ and $\nabla^{F}$, respectively. By Lemma \ref{lem Rh cpt}, we find that
\[
\prod_{g \in \Z} R^g_{\varphi, \nabla^F}(\sigma) = |R_{\varphi_{\R/\Z}, \rho}(\sigma)|^2,
\]
where $\rho$ is the unitary representation of $\pi_1(\R/\Z) = \Z$ defined by $e^{\alpha}$.  So we directly see that \eqref{eq R XM} holds in this example.

\section{Geodesic flow on $\R^n$}\label{sec Eucl}

The classical Ruelle $\zeta$-function is trivial for the geodesic flow on $\R^n \times S^{n-1} = S(T\R^n)$, because there are no closed geodesics. We show that this flow is $g$-nondegenerate
 for the action by the Euclidean motion group on $S(T\R^n)$, for group elements $g$ satisfying a regularity property. The expression for the equivariant Ruelle $\zeta$-function diverges, however. We show that for certain discrete subgroups of the Euclidean motion group, the equivariant Ruelle $\zeta$-function  does converge, and we compute it explicitly.
 In that case, there are no issues with convergence and meromorphic continuation, because the sums that occur are finite. We conclude in Proposition \ref{prop Fried Eucl} that the equivariant Fried conjecture is true in this case.

\subsection{Geodesic flow on $S(T\R^n)$}

Consider the natural action by the oriented Euclidean motion group $G = \R^n \rtimes \SO(n)$ on the Euclidean space $\R^n$, given by
\[
(y,r)\cdot x = rx+y
\]
for $x,y \in \R^n$ and $r \in \SO(n)$. Let $M = \R^n \times S^{n-1}$, viewed as the unit sphere bundle of $T\R^n$. 
For the rest of this section, the letters $x$ and $y$ denote elements of $\R^n$, $v$ denotes an element of $S^{n-1}$, $r$ an element of $\SO(n)$ and $t$ and $l$  real numbers.

The lift of the action to $M$ is given by
\[
(y,r)\cdot (x,v) = (rx+y, rv).
\]
 The exponential map of $\R^n$ is
\[
\exp_x(v) = x+v.
\]
 So the geodesic flow $\varphi$ on $M$ is given by
\[
\varphi_t(x,v) = \bigl(\exp_x(tv),\left. \frac{d}{ds}\right|_{s=t} \exp_x(sv)\bigr) = (x+tv, v).
\]
The generating vector field $u$ of this flow is
\beq{eq u Sn}
u(x,v) = (v,0).
\eeq
Consistently with Example 1 on page \pageref{ex geod flow invar}, the flow $\varphi$ is $G$-equivariant:
\[
\varphi_t((y,r)(x,v)) = (r(x+tv)+y, rv) = (y,r)\varphi_t(x,v).
\]

We denote the $n\times n$ identity matrix by $I$. 
\begin{lemma}\label{lem fixed w}
Let $r \in \SO(n)$, $x,y \in \R^n$, $v \in S^{n-1}$ and $l \in \R$. Then $\varphi_l(x,v) = (y,r)(x,v)$ if and only if $rv=v$ and
\beq{eq fixed w}
(r-I)w = lv-y,
\eeq
where $w$ is the component of $x$ orthogonal to $v$.  
\end{lemma}
\begin{proof}
Write $x = \lambda v + w$, with $\lambda \in \R$ and $w \perp v$. Then
\[
\begin{split}
\varphi_l(x,v) &= (\lambda v + w + lv, v);\\
(y,r)(x,v) &= (\lambda rv + rw + y, rv).
\end{split}
\]
The right hand sides are equal if and only if $rv=v$ and
\[
w+lv = rw+y.
\] 
The latter condition is equivalent to \eqref{eq fixed w}.
\end{proof}

%


\begin{lemma}\label{lem det 1-P}
Let $ g = (y,r) \in G$, for $y \in \R^n$ and $r \in \SO(n)$
 such that $\ker(r-I)$ is one-dimensional. Then the flow $\varphi$ is $g$-nondegenerate. 
\end{lemma}
\begin{proof}
Let $v_0 \in \ker(r-I)$ be a unit vector.  
 Let $(x,v) \in M$, and suppose that $\varphi_l(x,v) = g(x,v)$.  By Lemma \ref{lem fixed w}, we have $v = \pm v_0$, so by  \eqref{eq u Sn}, 
\[
u(x,v)^{\perp} = (v_0, 0)^{\perp} = v_0^{\perp} \times v_0^{\perp} \subset \R^n \times v_0^{\perp} = T_{(x,v)}M.
\]

Let $a \in  T_x \R^n = \R^n $ and $b \in T_vS^{n-1} \subset \R^n$. Then  a direct computation involving the equalities $rv=v$ and $rx+y = x+lv$ from Lemma \ref{lem fixed w} yields
\beq{eq P Eucl}
\begin{split}
T_{(x,v)} (\varphi_{-l} \circ g)(a,b) &= \left. \frac{d}{dt}\right|_{t=0} \bigl( r(x+ta)+y-lr(v+tb), r(v+tb)\bigr)\\
&=  (ra-lrb, rb).
\end{split}
\eeq
We find that, as a map from $v_0^{\perp} \oplus v_0^{\perp}$ to itself, 
\[
1 - T_{(x,v)} (\varphi_{-l} \circ g)|_{u(x,v)^{\perp}} = 
\left( 
\begin{array}{cc}
I-r & lr \\0 & I-r
\end{array}
\right),
\]
which implies that
\[
\det\left(1 - T_{(x,v)} (\varphi_{-l} \circ g)|_{u(x,v)^{\perp}} \right) = \det( (I-r)|_{v_0^{\perp}})^2.
\]
Since $I-r$ is invertible on $v_0^{\perp}$, the number $\det( (I-r)|_{v_0^{\perp}})^2 $ is positive.
%
\end{proof}
%

\begin{remark}
By a computation analogous to Lemma \ref{lem det 1-P}, the equivariant Ruelle $\zeta$-function is not well-defined for the action by the group $\R^n < G$ on $M$ if $n \geq 2$, because the flow is not $g$-nondegenerate for $g \in \R^n < G$.
\end{remark}

\begin{remark}
An element $r \in \SO(n)$ for which $\ker(r-I)$ is one-dimensional exists if and only if $n$ is odd. If $n=1$, this element $r$ is the identity element. So 
for 
nontrivial examples, $n=2k+1$ is an odd integer greater than or equal to $3$. In such cases, every element of $\SO(n)$ is conjugate to a block-diagonal element of $\SO(2)^k \times \{1\}$. The condition that $\ker(r-I)$ is one-dimensional means that all $k$ components in $\SO(2)$ are nontrivial, which is true for   elements of an open dense subset of  $\SO(n)$. This implies that the condition on $r$ holds generically.
\end{remark}

\subsection{Ingredients of the equivariant Ruelle $\zeta$-function}

\begin{lemma} \label{lem Lh}
Let $r \in \SO(n)$ be such that $\ker(r-I)$ is one-dimensional. Let $v_0$ be one of the two unit vectors in $\ker(r-I)$. Let $y \in \R^n$, and write
$g:= (y,r) \in G$. Then
\[
L_g(\varphi) = \{\pm l\},
\]
if $y = lv_0+w'$, for $w' \perp v_0$.
\end{lemma}
\begin{proof}
Let $l \in \Rnz$. Then $l \in L_g(\varphi)$ if and only if there are $x \in \R^n$ and $v \in S^{n-1}$ such that $\varphi_l(x,v) = g(x,v)$. By Lemma \ref{lem fixed w}, this is the case if and only if $rv=v$ and $(r-I)w = lv-y$, where $w$ is the component of $x$ orthogonal to $v$.

Since $\ker(r-I) = \{\pm v_0\}$, the condition $rv=v$ means $v = \pm v_0$. And then  $(r-I)w = lv-y$ means
\beq{eq y v0 w}
y = \pm l v_0 + (I-r)w.
\eeq
Write $y = \lambda v_0 + w'$, with $w' \perp v_0$. Then \eqref{eq y v0 w} becomes
\[
\lambda = \pm \l \quad \text{and} \quad w'=(I-r)w.
\]
As $\ker(r-I)$ is one-dimensional, $I-r$ is invertible on the orthogonal complement of $v_0$. So there is a unique $w$ such that $w'=(I-r)w$.
\end{proof}



Fix a group element $g = (l_0 v_0 + w', r)$, for $r \in \SO(n)$ with $\ker(r-I)$ one-dimensional, spanned by a unit vector $v_0$, $l_0 \in \Rnz$ and $w' \perp v_0$. By Lemma \ref{lem Lh}, we have $L_g(\varphi) = \{\pm l_0\}$.
\begin{lemma}\label{lem Gamma Eucl}
For $l = \pm l_0 \in L_g(\varphi)$, the set $\Gamma^g_l(\varphi)$ of $(g,l)$-periodic flow curves consists of a single flow curve of $\varphi$: the one through the point $(w, \pm v_0)$, for the unique $w \perp v_0$ such that $w' = (I-r)w$.
\end{lemma}
\begin{proof}
%
%
A point $(x,v)$ lies on a flow curve in $\Gamma^g_l(\varphi)$ precisely if $\varphi_l(x,v) = g(x,v)$. By Lemma \ref{lem fixed w}, this means that $rv=v$ and $x = \lambda v + w$, with $w \perp v$ such that \eqref{eq fixed w} holds.
The condition $rv=v$ means that $v = \pm v_0$. We distinguish the four cases $l = \pm l_0$ and $v = \pm v_0$.

\begin{itemize}
\item
If $l=l_0$ and $v=v_0$, 
then \eqref{eq fixed w} becomes
\[
(r-I)w = lv_0 -(lv_0+w') = -w'.
\]
This has a unique solution $w \perp v_0$, because $r-I$ is invertible on $v_0^{\perp}$.
\item
If $l=l_0$ and $v = -v_0$, then \eqref{eq fixed w} becomes
\[
(r-I)w = -lv_0 -(lv_0+w') = -2lv_0-w'.
\]
This has no solutions, because $l\not=0$ and $(r-I)w$ and $w'$ are orthogonal to $v_0$. 
\item If $l=-l_0$ and $v = v_0$, then there are no solutions, similarly to the second point.
\item  If $l=-l_0$ and $v = -v_0$, then analogously to the first point, we find a unique $w \perp v_0$ satisfying \eqref{eq fixed w}.
\end{itemize}

So if $l=l_0$, then
 $(x,v)$ lies on a flow curve in $\Gamma^g_l(\varphi)$ if and only if $v=v_0$ and $x = \lambda v_0 + w$, for the unique $w \perp v_0$ such that $(r-I)w=-w'$. Since
\[
(\lambda v_0 + w,v_0) = \varphi_{\lambda}(w,v_0),
\]
this means that $(x,v)$ lies on the unique flow curve through $(w,v_0)$.

The case where $l=-l_0$ is analogous. We then find that  the point $(x,v)$ lies on a flow curve in $\Gamma^h_l(\varphi)$ if and only if $v=-v_0$, and $x = \lambda v_0 + w$, for the unique $w \perp v_0$ such that $(r-I)w=-w'$. This means that $(x,v)$ lies on the flow curve through $(w,-v_0)$. 
\end{proof}


Let $F$ be the trivial line bundle on $M$. Consider the flat connection
\beq{eq conn Eucl}
\nabla^F := d + \sum_{j=1}^n \alpha_j dx_j + \beta
\eeq
on $F$, for $\alpha = (\alpha_1, \ldots, \alpha_n) \in \C^n$ and a closed one-form $\beta$ on $S^{n-1}$. We will later assume that $\nabla^F$ is invariant under actions by certain subgroups of $G$. 
\begin{lemma}\label{lem par transp}
Let $\gamma$ be a flow curve of $\varphi$, and write $\gamma(0) = (x,v)$. Then parallel transport along $\gamma$, from $t=0$ to $t=l$, with respect to $\nabla^F$, is given by multiplication by $e^{-l(\alpha, v)}$.
\end{lemma}
\begin{proof}
Let $s \in C^{\infty}(M)$, viewed as a section of $F$. Then for all $t \in \R$,
\[
(\nabla^F_{\gamma'}s)(\gamma(t)) = \langle d_{\gamma(t)} s, v\rangle + (\alpha, v)s(\gamma(t)).
\]
This equals zero if $s$ is given by
\[
s(x',v') = C e^{-(\alpha, x')},
\]
for a constant $C$, and $x' \in \R^n$ and $v' \in S^{n-1}$.
So 
the parallel transport map to be computed maps $s(x,v)$ to 
\[
s(\varphi_l(x,v)) = e^{-l(\alpha, v)}s(x,v).
\]
\end{proof}

If we let $G$ act trivially on the fibres of $F$, then Lemma \ref{lem par transp} implies that that  for the curves $\gamma_{\pm}$ in Lemma \ref{lem Gamma Eucl},
\beq{eq rho Eucl}
\rho_g(\gamma_{\pm}) = e^{l_0(\alpha, v_0)},
\eeq
independent of the sign.

\subsection{The equivariant Ruelle $\zeta$-function for a discrete subgroup}\label{sec Ruelle Eucl discr}

It turns out that for the action by the $G$ on $M$, the equivariant Ruelle $\zeta$-function \emph{diverges}. We compute this function for a discrete subgroup and an element $g$ for which it does converge.

\begin{lemma}\label{lem Ruelle diverge Eucl}
Suppose that $\nabla^F$ is $G$-invariant, so $\beta$ is $\SO(n)$-invariant and $\alpha = 0$. 
Let 
 $g = (l_0 v_0 + w', r) \in G$, for $r \in \SO(n)$ with $\ker(r-I)$ one-dimensional, spanned by a unit vector $v_0$, $l_0 \in \Rnz$ and $w' \perp v_0$. Then the right hand side of \eqref{eq Ruelle alt def} diverges, for the action by $G$ on $M$.
\end{lemma}
\begin{proof}
%
Let $\tilde \chi \in C^{\infty}_c(\R^n)$ be $\SO(n)$-invariant and such that
\[
\int_{\R^n} \tilde \chi(x)\, dx = 1.
\]
The pullback $\chi$ of $\tilde \chi$ to $M$ is a cutoff function for the action by $G$. So
by Lemmas \ref{lem Ruelle alt def}, \ref{lem det 1-P}, \ref{lem Lh},  \ref{lem Gamma Eucl} and equality \eqref{eq rho Eucl}, the right hand side of \eqref{eq Ruelle alt def}  equals
\[
\int_{G/Z} \frac{e^{-|l_0| \sigma }}{|l_0|}
 \int_{\R} \bigl( \chi(h\gamma_+(s)) + \chi(h \gamma_-(s)   ) \bigr)\, ds
 \, d(hZ),
\]
where we have used that $\alpha = 0$.
By right $Z$-invariance of the integrand in the definition of the Ruelle $\zeta$-function and compactness of $Z_{\SO(n)}(r)$, this equals
\[
\int_{G/\R v_0} \frac{e^{-|l_0| \sigma }}{|l_0|}
 \int_{\R} \bigl( \chi(h\gamma_+(s)) + \chi(h \gamma_-(s)   ) \bigr)\, ds
 \, d(h \R v_0).
 \]

Let $w \perp v_0$ be such that $(I-r)w = w'$. Then
\begin{align}
&\int_{G/\R v_0} \int_{\R}  \chi(h\gamma_{\pm}(s)) \, ds \, d(h \R v_0) \nonumber \\
&= 
\int_{\SO(n)}
 \int_{\R^n/\R v_0}  \int_{\R} 
\tilde \chi( r'w \pm tr' v_0 + x'  ) \, dt\, d(x'+\R v_0)\, dr' \nonumber \\
&= 
\int_{\SO(n)}
 \int_{v_0^{\perp}}  \int_{\R} 
\tilde \chi( r'w \pm tr' v_0 + x'  ) \, dt\, dx'\, dr'. \label{eq int Eucl}
\end{align}
For a given $r' \in \SO(3)$, 
define $\Phi\colon v_0^{\perp} \times \R \to \R^n$ by $\Phi(x', t) = tr'v_0 + x'$. This is a linear isomorphism if $r'v_0 \not \in v_0^{\perp}$. Its matrix representation with respect to the decomposition $\R^n = v_0^{\perp} \oplus \R v_0$ is
\[
\mat \Phi = \begin{pmatrix}
I|_{v_0^{\perp}} & r'v_0 -  (v_0, r'v_0) v_0 \\
0& (v_0, r'v_0)
\end{pmatrix}.
\]
The derivative of this map is the map itself,
and its determinant is $(v_0, r'v_0)$. So \eqref{eq int Eucl} equals
\[
\int_{\SO(n)} \frac{1}{|(v_0, r'v_0)|} \int_{\R^n} \tilde \chi(r'w + x)\, dx\, dr' = \int_{\SO(n)} \frac{1}{|(v_0, r'v_0)|} dr'.
\]
The integral on the right diverges.
\end{proof}

From now on, let $r \in \SO(n)$ be an element of finite order. We assume, as before, that $\ker(r-I)$ is one-dimensional, spanned by a unit vector $v_0$. Let $a>0$, and let $\Gamma'< v_0^{\perp}$ be an $r$-invariant,  discrete, cocompact subgroup. Consider the closed, cocompact subgroups 
\begin{align}
\Gamma &:= a \Z v_0 + \Gamma' < \R^n; \label{eq Gamma Eucl}\\
H &:= \Gamma \rtimes \langle r \rangle < G.\label{eq H Eucl}
\end{align}
Note that $\Gamma$ is invariant under $r$, so $H$ is well-defined. Consider the action by $H$ on $M$. 


Let $g = (al_0 v_0 + w', r^m)\in H$, for $l_0 \in \Z \setminus \{0\}$, $w' \in \Gamma'$ and $m \in \Z$. We suppose that $\ker(r^m - I)$ is one-dimensional, and hence equal to $\ker(r - I) = \R v_0$.
Let $Z<H$ be the centraliser of $g$.
\begin{lemma}\label{lem Z Eucl}
The map from $\Gamma'$ to $H/Z$ mapping $\gamma' \in \Gamma'$ to $(\gamma', I)Z$, is a bijection.
\end{lemma}
\begin{proof}
To determine $Z<H$, suppose that $\lambda,k \in \Z$, $\gamma' \in \Gamma'$,  and consider $(a\lambda v_0 + \gamma', r^k) \in H$. By a direct computation, we find that this element lies in $Z$ if and only if $(r^m - I)\gamma' = (r^k - I)w'$. So
\beq{eq Z Eucl}
Z = \bigl\{ \bigl( a\lambda v_0 + \bigl( (r^m - I)|_{v_0^{\perp}} \bigr)^{-1} (r^k - I)w', r^k \bigr); \lambda, k \in \Z \bigr\}.
\eeq
This implies the lemma.
\end{proof}

\begin{lemma}\label{lem Th Eucl}
If $\gamma_{\pm}$ are as in Lemma \ref{lem Gamma Eucl}, then for all $\varepsilon>0$, there is a cutoff function $\chi \in C^{\infty}_c(M)$ for the action by $\Gamma$ such that 
\beq{eq Th Eucl}
\left| 
\sum_{hZ \in H/Z} \int_{I_{\gamma_{\pm}}} \chi(h\gamma_{\pm}(s))\, ds - \frac{a}{\# \langle  r \rangle}
\right|<\varepsilon.
\eeq
\end{lemma}
\begin{proof}
Let $Y\subset M$ be an $r$-invariant fundamental domain for the action by $\Gamma$, such that the closure of $Y$ is compact, and equals the closure of the interior of $Y$. Define the function $\chi_0$ on $M$ by 
\[
\chi_0(m) = \left\{ 
\begin{array}{ll}
\frac{1}{\# \langle  r \rangle} & \text{if $m \in Y$};\\
0& \text{otherwise}. 
\end{array}
\right.
\]
This function is clearly not smooth, but it has the property \eqref{eq cutoff fn}, for the group $H$. 
By Lemma \ref{lem Z Eucl},
\beq{eq Th Eucl 1}
\sum_{hZ \in H/Z} \int_{I_{\gamma_{\pm}}} \chi_0(h\gamma_{\pm}(s))\, ds 
=
\sum_{\gamma' \in \Gamma'} \int_{\R}\chi_0\bigl(    \bigl((I-r)|_{v_0^{\perp}}\bigr)^{-1}w' + \gamma' \pm sv_0, \pm v_0 \bigr)\, ds.
\eeq
The point $\bigl(    \bigl((I-r)|_{v_0^{-1}}\bigr)^{-1}w' + \gamma' \pm sv_0, \pm v_0 \bigr)$ lies in $Y$ for a unique $\gamma' \in \Gamma'$, and for $s$ in a subset of $\R$ of length $a$. Hence the right hand side of \eqref{eq Th Eucl 1} equals
${a}/{\# \langle  r \rangle}$.

Since $\Gamma' \subset v_0^{\perp}$, the set $A \subset \Gamma'$ of $\gamma' \in \Gamma'$ for which there is $s \in \R$ such that
\[
(  \bigl((I-r)|_{v_0^{-1}}\bigr)^{-1}w' + \gamma' \pm s v_0, \pm v_0)
\]
lies within a distance $1$ of $Y$, 
is finite.
Now let $\varepsilon>0$. Let $\chi \in C^{\infty}_c(M)$ be a cutoff function such that  $\|\chi - \chi_0\|_{\infty}< \varepsilon$, and such that $\supp(\chi)$ lies within a distance $1$ of $Y$.  Then the left hand side of \eqref{eq Th Eucl} equals
\begin{multline}\label{eq Th Eucl 2}
\left|\sum_{hZ \in H/Z} \int_{I_{\gamma_{\pm}}} \chi(h\gamma_{\pm}(s))\, ds - \sum_{hZ \in H/Z} \int_{I_{\gamma_{\pm}}} \chi_0(h\gamma_{\pm}(s))\, ds  \right| \\
\leq 
\sum_{\gamma' \in A} \int_{\R} \bigl| (\chi - \chi_0) \bigl(    \bigl((I-r)|_{v_0^{-1}}\bigr)^{-1}w' + \gamma' \pm sv_0, \pm v_0 \bigr)  \bigr|\, ds.
\end{multline}
The function $\chi - \chi_0$ is supported within a distance $1$ from $Y$, so the integrand on the right is supported in a bounded interval $[-b,b]$. So the right hand side of \eqref{eq Th Eucl 2} is at most equal to $\# A \cdot 2b \cdot \varepsilon$. 
\end{proof}

Consider the connection \eqref{eq conn Eucl} on the trivial line bundle on $M$. Now suppose that $\nabla^F$ is $H$-invariant; this means that $\beta$ is $r$-invariant and $\alpha \in \C v_0$. 
\begin{proposition}\label{prop Ruelle Eucl}
For all $\sigma \in \C$, 
\[
 R^g_{\varphi, \nabla^F}(\sigma) = \exp\left( \frac{2 a e^{-|l_0| \sigma + l_0 (\alpha, v_0)}}{\# \langle r \rangle  |l_0|} \right).
\]
\end{proposition}
\begin{proof}
Let $\chi \in C^{\infty}_c(M)$ be a cutoff function.
By
 Lemmas 
 \ref{lem det 1-P},  \ref{lem Lh} and \ref{lem Gamma Eucl} and equality \eqref{eq rho Eucl}, we find that the right hand side of \eqref{eq Ruelle alt def}  equals
\[
\sum_{hZ \in \Gamma/Z} \frac{e^{-|l_0| \sigma + l_0 (\alpha, v_0)}}{|l_0|}
 \int_{\R} \bigl( \chi(h\gamma_+(s)) + \chi(h \gamma_-(s)   ) \bigr)\, ds.
\]
Hence the claim follows from Lemma \ref{lem Th Eucl} and independence of $R^g_{\varphi, \nabla^F}$ of the cutoff function (Proposition \ref{prop R indep chi}).
\end{proof}
%
%

\begin{remark}\label{rem conv Rn}
In this example, Lemma \ref{lem Ruelle subgroup} does not apply, since \eqref{eq Ruelle finite L} does not converge absolutely. We have seen that the conclusion of the lemma is not true: the Ruelle $\zeta$ function diverges for $G$ (Lemma \ref{lem Ruelle diverge Eucl}), but converges for the cocompact subgroup $H$ (Proposition \ref{prop Ruelle Eucl}).
\end{remark}

\subsection{Equivariant analytic torsion}

In this subsection, we compute equivariant analytic torsion in the setting of Subsection \ref{sec Ruelle Eucl discr}, and conclude that the answer to Question \ref{q Fried} is \emph{yes}.

Let $r \in \SO(n)$ be as in Subsection \ref{sec Ruelle Eucl discr}, i.e.\ of finite order and with $\ker(r-I) = \R v_0$. Consider the closed,  cocompact subgroup
\[
G' := \R^n \rtimes \langle r \rangle < G.
\]
It decomposes as a Cartesian product:
\beq{eq prod G'}
G' = \R v_0 \times \bigl( v_0^{\perp} \rtimes  \langle r \rangle  \bigr).
\eeq
Furthermore, the action by $G'$ on $M$ is the Cartesian product of the action by $\R v_0$ on iself and the action by $v_0^{\perp} \rtimes  \langle r \rangle$ on $v_0^{\perp} \times S^{n-1}$. This means that we can use the product formula for equivariant analytic torsion, Proposition 5.8 in \cite{HS22a} to compute the equivariant analytic torsion for the action by  $G'$ on $M$.

This product formula involves an equivariant Euler characteristic, see Definition 5.1 in \cite{HS22a}. Suppose that a group $B$ acts properly, cocompactly and isometrically on a manifold $N$. Let $b \in B$, and fix a measure $dz$ on its centraliser $Z$. Let $N^b$ be the fixed-point set of $b$. Let $\chi_Z \in C^{\infty}(N^b)$ be such that for all $n \in N^b$, 
\beq{eq cutoff Z}
\int_Z \chi_Z(zn)\, dz = 1.
\eeq
Then the equivariant Euler characteristic of $N$ with respect to $b$ is
\[
\chi_b(N) :=
\frac{1}{(2\pi)^{\dim(N^b)/2}} \int_{N^b} \chi_Z \Pf(-R^{N^b}).
\]
Here $R^{N^b}$ is the curvature of the Levi-Civita connection on $N^g$, and $\Pf$ denotes the Pfaffian. (Note that the letter $\chi$ is part of the notation for both 
 the cutoff function and the Euler characteristic.)

Let $g = (l_0 v_0 + w', r^m) \in G'$, for $l_0 \in \Rnz$, $w' \in v_0^{\perp}$ and $m \in \Z$. 
\begin{lemma}\label{lem chi g Eucl}
We have
\beq{eq chi g Eucl}
\chi_{ (w', r^m)} (v_0^{\perp} \times S^{n-1}) = \frac{2}{\# \langle r \rangle}.
\eeq
\end{lemma}
\begin{proof}
%
%
We determine the fixed point set $(v_0^{\perp} \times S^{n-1})^{(w', r^m)}$ and the centraliser $Z<G'$. If $x \in v_0^{\perp}$ and $v \in S^{n-1}$, then the point $(x, v)$ is fixed by $(w', r^m)$ if and only if
\[
\begin{split}
(I - r^m) v &= 0;\\
(I-r^m)x &= w'.
\end{split}
\]
So
\beq{eq fixed Eucl}
(v_0^{\perp} \times S^{n-1})^{(w', r^m)} = \bigl\{
\bigl(  
\bigl(  
(I-r^m)|_{v_0^{\perp}}
 \bigr)^{-1}w', \pm v_0
 \bigr)
 \bigr\}
\eeq
consists of two points.

To determine $Z$, let $u \in v_0^{\perp}$ and $k \in \Z$. Then $(u, r'k) \in Z$ if and only if $(I-r^m)u = (I-r^k)w'$. So
\[
Z = 
\bigl\{
\bigl(  
\bigl(  
(I-r^m)|_{v_0^{\perp}}
 \bigr)^{-1}(I-r^k)w', r^k
 \bigr); k \in \Z
 \bigr\}
\] 
is cyclic of order $\# \langle r\rangle$. So the constant function $\chi_Z$ equal to $\frac{1}{\# \langle r\rangle}$ has the property \eqref{eq cutoff Z}. Using this function and  \eqref{eq fixed Eucl}, we find that \eqref{eq chi g Eucl} holds.
\end{proof}

Let $\nabla^F$ be the connection \ref{eq conn Eucl} on the trivial line bundle on $M$, now with $\alpha \in i\R^n$. We assume that $\nabla^F$ is $G'$-invariant, which means that $\beta$ is $r$-invariant and $\alpha \in i\R v_0$. 
\begin{proposition} \label{prop torsion Eucl G'}
The equivariant analytic torsion $T_g(\nabla^F)$ for the action by $G'$ on $M$ converges, and is given by
\beq{eq torsion Eucl G'}
 T_g(\nabla^F) = \exp\left( \frac{e^{l_0(\alpha, v_0)}}{\# \langle r \rangle |l_0|} \right).
\eeq
\end{proposition}
\begin{proof}
For $t>0$,
let $\cT^M_g(t)$ be as in \eqref{eq curly T}, for the action by $G'$ on $M$. Let $\cT^{\R v_0}_{l_0 v_0}(t)$ and $\cT_{(w', r^m)}^{v_0^{\perp} \times S^{n-1}}(t)$ be the analogous numbers for the action by $\R v_0$ on itself and  the action by $v_0^{\perp} \rtimes  \langle r \rangle$ on $v_0^{\perp} \times S^{n-1}$, respectively. Because of the product structure \eqref{eq prod G'} of  $G'$ and its action  on $M$, the proof of Proposition 5.8 in \cite{HS22a} implies that for all $t>0$
\beq{eq curly T Eucl}
\cT^M_g(t) = \chi_{ (w', r^m)} (v_0^{\perp} \times S^{n-1}) \cT^{\R v_0}_{l_0 v_0}(t) + \chi_{l_0 v_0}(\R v_0) \cT_{(w', r^m)}^{v_0^{\perp} \times S^{n-1}}(t).
\eeq
Now $\chi_{l_0 v_0}(\R v_0) = 0$, because $(\R v_0)^{l_0 v_0} = \emptyset$. So by Lemma \ref{lem chi g Eucl}, the right hand side of \eqref{eq curly T Eucl} equals
\[
\frac{2}{\# \langle r \rangle }\cT^{\R v_0}_{l_0 v_0}(t). 
\] 
The number on the right is computed with respect to the connection $d + (\alpha, v_0) dx$ on $\R v_0\cong \R$. For that connection, with $ (\alpha, v_0)$ imaginary,  it was shown in Proposition 6.2 in \cite{HS22a} that  equivariant analytic torsion converges, and is given by
\[
\log 
T_{l_0} (d+  (\alpha, v_0) dx) = 
 \frac{e^{ l_0  (\alpha, v_0)}}{2|l_0|}. 
\]
We conclude that $T_g(\nabla^F)$ converges, and 
\[
\log T_g(\nabla^F) = \frac{2}{\# \langle r \rangle } \log T_{l_0} (d+ (\alpha, v_0) dx),
\]
which is  \eqref{eq torsion Eucl G'}.
\end{proof}

We can now explicitly verify the equivariant Fried conjecture in this setting.
\begin{proposition}\label{prop Fried Eucl}
The kernel of the Laplacian on $M$ associated to the connection $\nabla^F$ is trivial.
Let $g$ be as above Lemma \ref{lem Z Eucl}. Suppose that $\alpha$ in \eqref{eq conn Eucl} is imaginary. 
 Let $H< G'$ be as in \eqref{eq H Eucl}. For the action by $H$ on $M$, 
\beq{eq Fried Eucl}
R^g_{\varphi, \nabla^F}(0) = T_g(\nabla^F)^2.
\eeq
\end{proposition}
\begin{proof}
The $L^2$-kernel of the Laplacian on $M$ is the tensor product of the $L^2$-kernels of the Laplacians on $\R^n$ and on $S^{n-1}$. The former of these is zero, hence so is the $L^2$-kernel of the Laplacian on $M$.

Let $Z_{G'}$ be the centraliser of $g$ in $G'$, and $Z_H$  the centraliser of $g$ in $H$. Let $T_g^{G'}(\nabla^F)$ be the analytic torsion of $\nabla^F$ with respect to the action by $G'$. 
Proposition 2.4 in \cite{HS22a} states that  the analytic torsion of $\nabla^F$ with respect to the action by $H$ equals
\beq{eq torsion subgp Eucl}
T_g(\nabla^F) = T_g^{G'}(\nabla^F)^{\vol(Z_{G'}/Z_H)}.
\eeq
The centraliser $Z_{H}$ is given by \eqref{eq Z Eucl}. By a similar argument, 
\[
Z_{G'} = \bigl\{ \bigl( a\lambda v_0 + \bigl( (r^m - I)|_{v_0^{\perp}} \bigr)^{-1} (r^k - I)w', r^k \bigr); \lambda \in \R, k \in \Z \bigr\}.
\]
So $\vol(Z_{G'}/Z_H) = \vol(\R/a\Z) = a$. Hence by \eqref{eq torsion subgp Eucl} and Proposition \ref{prop torsion Eucl G'},
\[
\log T_g(\nabla^F) = a \log T_g^{G'}(\nabla^F) =  \frac{a e^{l_0 (\alpha, v_0)}}{ \# \langle r \rangle |l_0|}.
\]
So \eqref{eq Fried Eucl} follows from Proposition \ref{prop Ruelle Eucl}.
\end{proof}

\section{Geodesic flow on spheres}\label{sec Ruelle S2}

For geodesic flow on the sphere bundle on a sphere, the classical Ruelle $\zeta$-function is not defined, because there are uncountably many closed geodesics. This flow is not Anosov (which would imply that the Ruelle $\zeta$-function is well-defined), in contrast to the geodesic flow in negative curvature. In this section, 
we onsider the action by $\SO(n+1)$ on the sphere bundle  $S(TS^n)$. For $n=2$ and $n=3$, we show that this flow is $g$-nondegenerate for suitable group elements, and compute the equivariant Ruelle $\zeta$-function for the trivial connection on the trivial line bundle. We expect these computations to generalise to arbitrary $n$, with some additional combinatorics involved. In these cases, the kernel of $\Delta_F$ is nonzero, so the hypotheses of the equivariant Fried problem, Question \ref{q Fried}, are not satisfied.


\subsection{Homogeneous spaces of compact Lie groups} \label{sec hom sp}

We assume in this subsection that $G$ is a  Lie group, and $H<G$ is a closed subgroup. Consider the action by $G\times H$ on $G$ given by
\[
(x,h) \cdot y = xyh^{-1},
\]
for $x,y \in G$ and $h \in H$.
Let $X_0 \in \kg$ be such that $\ad(X_0)|_{\kh} = 0$. For $t \in \R$, we define $\tilde \varphi_t\colon G \to G$ by 
\[
\tilde \varphi_t(x) = x\exp(tX_0),
\]
for $x \in G$. Because $\tilde \varphi_t$ is $H$-equivariant, it descends to a map $\varphi_t\colon G/H \to G/H$.
\begin{example}
Geodesic flow on the sphere bundle of a sphere occurs in this way, see Lemma \ref{lem geod flow Sn}.
\end{example}
Fix $g \in G$. If $y \in G$, and $V \subset \kg$ is a linear subspace, then we denote the map from $\kg/V$ to $\kg/\Ad(y)V$ induced by $\Ad(y)$ by $\Ad(y)$ as well. 
\begin{lemma}\label{lem phi Ad}
Let $l \in \R$ and $x \in G$ be such that $\varphi_l(xH) = gxH$. Then the following diagram commutes:
\[
\xymatrix{
T_{xH}G/H \ar[rrr]^-{T_{xH}(\varphi_l \circ g^{-1})} &&& T_{xH}G/H\\
\kg/ \Ad(x) \kh \ar[rrr]^-{\Ad(g^{-1})} \ar[u]^-{T_{eH}l_x \circ \Ad(x)^{-1}}&&& \kg/ \Ad(x)\kh \ar[u]_-{T_{eH}l_x \circ \Ad(x)^{-1}}
}
\]
\end{lemma}
\begin{proof}
Let $l$ and $x$ be as in the statement of the lemma. 
Since the map $\varphi_l$ is $G$-equivariant, 
\beq{eq phi Ad 1}
T_{eH}l_{x^{-1}} \circ T_{xH}(\varphi_l \circ g^{-1}) \circ T_{eH}l_{x} = T_{eH} (\varphi_l \circ l_{x^{-1}g^{-1}x})\colon \kg/\kh \to \kg/\kh.  
\eeq
The equality $\varphi_l(xH) = gxH$ implies that $\exp(lX_0)H = x^{-1}gxH$. So for all $y \in G$, 
\[
(\varphi_l \circ l_{x^{-1}g^{-1}x})yH = (x^{-1}g^{-1}x) y (x^{-1}g x)H.
\]
It follows that the right hand side of \eqref{eq phi Ad 1} is $\Ad(x^{-1}g^{-1}x)$, and the lemma follows.
\end{proof}




\subsection{Geodesic flow on $S(TS^n)$}

Let $M = S(TS^n)$, acted on by $G = \SO(n+1)$ in the natural way. Let $\varphi$ be the geodesic flow on $M$. By Example 1 on page \pageref{ex geod flow invar}, it is 
equivariant. 

Consider the subspace
\[
\km := \left\{ \begin{pmatrix}
0 & -x_1 & \cdots & -x_n \\
x_1 & 0  & \cdots & 0 \\
\vdots & \vdots &  & \vdots \\
x_n & 0 & \cdots & 0
\end{pmatrix}; x_1, \ldots, x_n \in \R
\right\} \subset \kso(n+1).
\]
Let $S(\km) \subset \km$ be the unit sphere with respect to the Euclidean metric on $\km \cong \R^n$.
For a given $p \in S^n$, we have the natural $\SO(n+1)$-equivariant diffeomorphism
\beq{eq diffeo Sn}
 \SO(n+1) \times_{\SO(n)} S(\km) \to M 
 \eeq
 mapping $[x,X]$ to 
\[
\ddt x \exp(tX) p = xXp,
\]
for $x \in \SO(n+1)$ and $X \in S(\km)$. 

The adjoint action by $\SO(n)$ on $\km  \cong \R^n$ is the standard action by $\SO(n)$. So for a given $X_0 \in S(\km)$, the action by $\SO(n)$ on $X_0$ defines an $\SO(n)$-equivariant diffeomorphism $\SO(n)/\SO(n-1) \to S(\km)$. Combining this with \eqref{eq diffeo Sn} yields an $\SO(n+1)$-equivariant diffeomorphism
\beq{eq diffeo Sn 2}
\SO(n+1)/\SO(n-1) \to M.
\eeq
So $M$ is of the form $G/H$ in Subsection \ref{sec hom sp}, with $H  = \SO(n-1)$. Geodesic flow on $M$ is of the form $\varphi_t$ in that subsection.
%
\begin{lemma}\label{lem geod flow Sn}
Under the diffeomorphism \eqref{eq diffeo Sn 2}, the geodesic flow on $M$ is given by
\beq{eq Riem exp Sn hom 2}
\varphi_t(x \SO(n-1)) = x\exp(tX_0)\SO(n-1),
\eeq
for $t \in \R$ and $x \in \SO(n+1)$. 
\end{lemma}
\begin{proof}
We have $[\km, \km] \subset \kso(n)$, so $S^n = \SO(n+1)/\SO(n)$ is a Riemannian symmetric space. The Riemannian exponential map is therefore given by
\beq{eq symm exp}
\exp_{e \SO(n)}(X) = \exp(X)\SO(n),
\eeq
for $X \in \km \cong T_{e \SO(n)} (\SO(n+1)/\SO(n))$. See e.g.\ page 832 of \cite{Helgason58}. Here the map $\exp_{e\SO(n)}$ on the left hand side is the Riemannian exponential map, and the map $\exp$ on the right hand side is the Lie-theoretic one.
Then \eqref{eq symm exp} and equivariance of the Riemannian exponential map imply 
that
\beq{eq flow S2}
\varphi_t(p, Xp) = (\exp(tX)p, \exp(tX)Xp)
\eeq
for $t \in \R$, $p \in S^n$ and $X \in S(\km)$. 
%
%
An application of the diffeomorphism \eqref{eq diffeo Sn} then implies that 
\[
\varphi_t([x,X]) = [x \exp(tX), X]
\]
for  $x \in \SO(n+1)$. Finally,  \eqref{eq Riem exp Sn hom 2} follows by an application of \eqref{eq diffeo Sn 2}.
%
%
%
\end{proof}

%


From now on, we make the specific choice
\[
X_0 =  \begin{pmatrix}
0 & -1 & 0& \cdots & 0 \\
1 & 0  & \cdots & \cdots & 0 \\
0 & \cdots && \cdots &  0 \\
\vdots & &  &  & \vdots \\
0 &\cdots & & \cdots & 0
\end{pmatrix} \in S(\km). 
\]

We write $n=2k$ if $n$ is even, and $n = 2k-1$ if $n$ is odd. Then $T:= \SO(2)^k$ embeds into $\SO(n+1)$ as a maximal torus. If $n$ is odd, then we embed $T$ into $\SO(n+1)$ as diagonal $2\times 2$ blocks. If $n$ is even, we embed $T$ into $\SO(n+1)$ as diagonal $2\times 2$ blocks with a $1$ in the bottom-right entry.
For $\theta \in \R$, we write
\[
r(\theta) := \begin{pmatrix}
\cos(\theta) & -\sin(\theta)\\
\sin(\theta) & \cos(\theta)
\end{pmatrix} \in \SO(2).
\]
Let $\theta_1, \ldots, \theta_k \in \R$, and suppose that
 $g = (r(\theta_1), \ldots, r(\theta_k)) \in T$.
\begin{lemma}\label{lem Lg SOn}
We have
\[
L_g(\varphi) \subset \bigcup_{j=1}^k \bigl(\theta_j + 2\pi \Z \cup -\theta_j + 2\pi \Z\bigr).
\]
\end{lemma}
\begin{proof}
Suppose that $l \in L_g(\varphi)$. 
By Lemma \ref{lem geod flow Sn}, the condition $\varphi_l(x\SO(n-1)) = gx\SO(n-1)$ is equivalent to $x^{-1}g^{-1}x \exp(lX_0) \in \{I_2\} \times \SO(n-1)$.  So there is an $h \in \SO(n-1)$ such that
\beq{eq x g lX0}
x^{-1}gx  = 
\begin{pmatrix}
\exp(lX_0) & 0 \\
0 & h
\end{pmatrix}.
\eeq
In particular, the matrix on the right has the same eigenvalues as $g$. This implies that the eigenvalues of $\exp(lX_0)$, which are $e^{\pm i l}$, equal $e^{\pm i\theta_j}$ for some $j$. Hence $l \in \pm {\theta_j}+ 2\pi \Z$.
%
%
%
%
\end{proof}

\begin{remark}
Lemma \ref{lem Lg SOn} implies that $L_g(\varphi)$ is countable. The set of  $l \in \Rnz$ for which there exists a geodesic $\gamma$ on $S^n$ such that $\gamma(l) = g\gamma(l)$ is uncountable, because the set
\[
\{d(m, gm); m \in S^n\}
\]
is uncountable. This shows that $L_g(\varphi)$ is different from the latter set. The difference is caused by the action by $g$ on unit tangent vectors. 
\end{remark}

%
%


\begin{lemma}
If $n\geq 2$, then 
the only $G$-invariant, flat, Hermitian connection on  the trivial line bundle $F = M \times \C \to M$ is the trivial connection $d$.
\end{lemma}
\begin{proof}
Let  $\nabla^F$ be a $G$-invariant, Hermitian connection on  $F$. Then  $\nabla^F = d+ i\omega$ for an element 
 $\omega\in \kso(n)^*$, viewed as a left-invariant one-form on $M \cong \SO(n)$. For $\nabla^F$ to be flat, one needs $d\omega = 0$. This is the case if and only if $\omega$ is zero  on $[\kso(n), \kso(n)] = \kso(n)$.
\end{proof}

\subsection{The case $n=2$}

In the rest of this section,  we assume that the powers of $g$ are dense in $T$.

Suppose for this subsection that $n=2$. Then $M = G = \SO(3)$, and $H$ is trivial. Since $g$ is regular, it follows from Lemma \ref{lem phi Ad} that for all $l \in \Rnz$ and $x \in G$  such that $\varphi_l(xH) = gxH$, 
\[
\dim(\ker(T_{xH} (\varphi_l \circ g^{-1}) - 1)) = \dim(\ker(\Ad(g^{-1})-1)) = 1.
\]

So the flow $\varphi$ is $g$-nondegenerate.

\begin{lemma}\label{lem Lg S2}
Suppose that $g$ is a rotation by an angle   $\theta$.
 Then 
 \beq{eq Lg S2}
 L_g(\varphi) = \theta + 2\pi \Z \cup  -\theta + 2\pi \Z.
 \eeq
\end{lemma}
\begin{proof}
By  Lemma \ref{lem Lg SOn}, it is enough to show that the right hand side of \eqref{eq Lg S2} is contained in the left hand side.
If $l \in \R$ and $x \in \SO(3)$, then $\varphi_l(x) = gx$ if and only if $x^{-1}gx = \exp(lX_0)$. If $l \in \theta + 2\pi \Z$, then $\exp(lX_0) = g$, so taking $x=e$ shows that $l \in L_g(\varphi)$. And if $l \in -\theta + 2\pi \Z$, then $\exp(lX_0) = g^{-1}$, and taking
\[
x = \begin{pmatrix}
0 & 1 & 0 \\
1 & 0 & 0 \\
0 & 0 & -1
 \end{pmatrix}
\]
shows that $l \in L_g(\varphi)$. 
\end{proof}

\begin{lemma}\label{lem Gamma g S2}
Let $g \in \SO(3)$ be a rotation over an angle  $\theta$. For all $l \in L_g(\varphi)$, the set $\Gamma_l^g(\varphi)$ consists of a single flow curve.
\end{lemma}
\begin{proof}
Let $x,x' \in \SO(3) \cong M$, and suppose that $\varphi_l(x) = gx$ and  $\varphi_l(x') = gx'$. 
Then $x^{-1}gx = \exp(lX_0) = x'^{-1}gx'$.
Since the powers of $g$ are dense in $\SO(2)$, it follows that for all $y \in \SO(2)$, 
\[
x^{-1}yx = x'^{-1}yx'.
\]
This implies that $x^{-1}x'$ lies in the centraliser of $\SO(2)$, which is $\SO(2)$ itself.
We therefore find that
\[
x' = x\exp(tX_0)
\]
for some $t \in \R$, so $x'$ and $x$ lie on the same flow curve of $\varphi$. 
\end{proof}

\begin{remark} In Lemma \ref{lem Gamma g S2} (and also in Lemma \ref{lem Gamma g S3} below), it is crucial that the powers of $g$ are dense in $T$. In the opposite extreme case where $g$ is the identity element, the set $L_e(\varphi) = 2\pi \Z \setminus \{0\}$ is still countable, but $\Gamma_l^e(\varphi)$ is uncountable for all $l \in 2\pi \Z \setminus \{0\}$ if $n \geq 2$. This is  because there is a closed geodesic though every point in $S^n$. This is  why the classical Ruelle $\zeta$-function is not defined for the geodesic flow on spheres, whereas the equivariant version is defined for $g$ with dense powers in $\SO(2)$, as we will see for $n=2$ and $n=3$.
\end{remark}

\begin{proposition}\label{prop Ruelle S2}
Let $F = M \times \C \to M$ be the trivial line bundle, and $\nabla^F = d$ the trivial connection.
The flow $\varphi$ is $g$-nondegenerate, the set $L_g(\varphi)$ is countable, and  for all $\sigma \in \C$ with positive real parts, 
\beq{eq Ruelle S2}
\log R^g_{\varphi, \nabla^F}(\sigma) = 2\pi \sum_{n \in \Z} \left( \frac{e^{  -|\theta + 2\pi n| \sigma}}{|\theta + 2\pi n|}+
\frac{e^{  -|-\theta + 2\pi n| \sigma}}{|-\theta + 2\pi n|}
\right)
.
\eeq
\end{proposition}
\begin{proof}
In this case, $\sgn(\det(1-P_{\gamma}^{g})) = 1$ for all $l \in L_g(\varphi)$ and $\gamma \in \Gamma_g^l(\varphi)$,  because $S^2$ is oriented. The constant function $1$ is a cutoff function because $G$ is compact, while one can take $I_{\gamma} = [0, 2\pi]$ for any $g$-periodic flow curve $\gamma$. So for all such curves $\gamma$ and $x \in \SO(3)$, 
\[
\int_{I_{\gamma}} \chi(x\gamma(s))\, ds = 2\pi.
\]
%
%

Normalising the measure on $G/Z$ to give it unit volume, and applying Lemmas \ref{lem Ruelle alt def},  \ref{lem Lg S2} and \ref{lem Gamma g S2}, we conclude that \eqref{eq Ruelle S2} holds. 

\end{proof}

%

In this example, the kernel of $\Delta_F$ is the cohomology of $\SO(3)$, which is not zero. Furthermore, the  equivariant Ruelle $\zeta$-function does not extend to $\sigma = 0$;  Lemma \ref{lem mer cont gen} does not apply, since $\alpha = 0$ in the current case. So we are not in the situation of the equivariant Fried conjecture.

\subsection{The case $n=3$}

In this subsection, we suppose that $n=3$, and remind the reader that we are assuming the powers of $g$ are dense in $T$.
Then $H = \{I_2\} \times \SO(2)$ is the copy of $\SO(2)$ embedded into $\SO(4)$ in the bottom right corner.

Consider the matrices
\[
w_1 := I_2, \quad w_{-1} := \begin{pmatrix} 0 & 1 \\ 1 & 0
\end{pmatrix} \quad \in \OO(2). 
\]
Then for all $\varepsilon \in \{\pm 1\}$ and $\alpha \in \R$,
\beq{eq w eps}
r(\varepsilon \alpha) = w_{\varepsilon} r(\alpha) w_{\varepsilon}.
\eeq
\begin{lemma}\label{lem Lg S3}
We have
\beq{eq Lg S3}
L_g(\varphi) = (\theta_1 + 2\pi \Z) \cup
 (-\theta_1 + 2\pi \Z) \cup
  (\theta_2 + 2\pi \Z) \cup
   (-\theta_2 + 2\pi \Z).
\eeq
Let $\varepsilon \in \{\pm 1\}$. 
If $l \in \varepsilon \theta_1 + 2\pi \Z$ and $x \in \SO(4)$, then the condition $\varphi_l(xH) = gxH$ is equivalent to 
\beq{eq Lg S3 3}
x = \begin{pmatrix}   aw_{\varepsilon} & 0 \\ 0 &   d w_{\varepsilon}\end{pmatrix},
\eeq
 for $  a,  d \in \SO(2)$. If $l \in \varepsilon \theta_2 + 2\pi \Z$ and $x \in \SO(4)$, then the condition $\varphi_l(xH) = gxH$ is equivalent to 
\beq{eq Lg S3 4}
x = \begin{pmatrix} 0 &   b w_{\varepsilon} \\    c w_{-\varepsilon} & 0\end{pmatrix},
\eeq
 for $  b,  c \in \SO(2)$.
\end{lemma}
\begin{proof}
By Lemma \ref{lem Lg SOn}, the equality \eqref{eq Lg S3} follows if the right hand side is contained in the left hand side. This follows from the characterisations \eqref{eq Lg S3 3} and \eqref{eq Lg S3 4} of the elements $x \in \SO(4)$ such that $\varphi_l(xH) = gxH$, for $l$ in the right hand side of \eqref{eq Lg S3}. We will prove these characterisations.

Let $l \in \Rnz$ and $x \in \SO(4)$ be such that $\varphi_l(xH) = xgH$. As in the proof of Lemma \ref{lem Lg SOn}, this is equivalent to
\beq{eq Lg S3 fixed}
xgx^{-1} = \begin{pmatrix} r(l) & 0 \\ 0 & r(\alpha)
\end{pmatrix},
\eeq
where $l = \pm \theta_1$ and $\alpha = \pm \theta_2$, or 
$l = \pm \theta_2$ and $\alpha = \pm \theta_1$, where all equalities (here and also below) are modulo $2\pi \Z$. It is immediate that if $x$ is of the form \eqref{eq Lg S3 3} and \eqref{eq Lg S3 4} in the respective cases, then \eqref{eq Lg S3 fixed} holds. So it remains to prove the converse implication.

Since, by assumption, the powers of $g$ are dense in $T$,  the equality \eqref{eq Lg S3 fixed} implies that  for all $\alpha_1, \alpha_2 \in \R$,
\beq{eq Lg S3 1}
x^{-1}
\begin{pmatrix} r(\alpha_1) & 0 \\ 0 & r(\alpha_2)
\end{pmatrix} x= 
\begin{pmatrix} r(l) & 0 \\ 0 & r(\alpha)
\end{pmatrix}
\eeq
for $l = \pm \alpha_1$ and $\alpha = \pm \alpha_2$, or 
$l = \pm \alpha_2$ and $\alpha = \pm \alpha_1$. 

Write $x = \begin{pmatrix} a & b \\ c & d\end{pmatrix}$, for $a,b,c,d \in M_2(\R)$. Then the left hand side of \eqref{eq Lg S3 1} equals
\beq{eq Lg S3 2}
\begin{pmatrix} a^T r(\alpha_1)a + c^T r(\alpha_2)c & a^T r(\alpha_1) b + c^Tr(\alpha_2) d \\ 
b^T r(\alpha_1) a + d^T r(\alpha_2 )c & b^T r(\alpha_1) b + d^T r(\alpha_2) d\end{pmatrix}.
\eeq

We first consider the case where $l = \pm \alpha_1$ and $\alpha = \pm \alpha_2$. Let $\varepsilon_1, \varepsilon_2 \in \{\pm 1\}$ be such that $l = \varepsilon_1 \alpha_1$ and $\alpha = \varepsilon_2 \alpha_2$. 
Then by \eqref{eq w eps}, we have  $r(l) = w_{\varepsilon_1} r(\alpha_1) w_{\varepsilon_1}$ and $r(\alpha) = w_{\varepsilon_2} r(\alpha_2) w_{\varepsilon_2}$. Since the diagonal entries of \eqref{eq Lg S3 2} equal the diagonal entries of the right hand side of \eqref{eq Lg S3 1}, we have for all $\alpha_1, \alpha_2 \in \R$, 
\beq{eq r alpha}
\begin{split}
a^T r(\alpha_1)a + c^T r(\alpha_2)c &= w_{\varepsilon_1} r(\alpha_1) w_{\varepsilon_1};\\
b^T r(\alpha_1) b + d^T r(\alpha_2) d&= w_{\varepsilon_2} r(\alpha_2) w_{\varepsilon_2}.
\end{split}
\eeq
So $b=c=0$, 
and $x = \begin{pmatrix} a & 0 \\ 0 & d\end{pmatrix}$, for $a,d \in \OO(2)$. 
The products $\tilde a := aw_{\varepsilon_1}$ and $\tilde d := d w_{\varepsilon_2}$ both centralise $\SO(2)$  by \eqref{eq r alpha}, so they lie in $\SO(2)$.
 We find that $x = \begin{pmatrix} \tilde aw_{\varepsilon_1} & 0 \\ 0 & \tilde d w_{\varepsilon_2}\end{pmatrix}$, for $\tilde a,\tilde d \in \SO(2)$. Since $x$ has determinant $1$, we have $\varepsilon_1 = \varepsilon_2$. 

In the case where $l = \varepsilon_1 \alpha_2$ and $\alpha = \varepsilon_2 \alpha_1$ for $\varepsilon_1, \varepsilon_2 \in\{\pm 1\}$, we can argue similarly. Then the equality between \eqref{eq Lg S3 2} and the right hand side of \eqref{eq Lg S3 1} is equivalent to $a=d=0$, and $x = \begin{pmatrix}  0 & \tilde b w_{\varepsilon_2} \\  \tilde c w_{\varepsilon_1} & 0\end{pmatrix}$ with $\tilde b := b w_{\varepsilon_2}$ and $\tilde c := cw_{\varepsilon_1}$ lie in $\SO(2)$. The fact that $x$ has determinant $1$ now implies that $\varepsilon_1 = -\varepsilon_2$.
\end{proof}

\begin{lemma}
The flow $\varphi_l$ is $g$-nondegenerate.
\end{lemma}
\begin{proof}
Let $l \in L_g(\varphi)$ and $x \in \SO(4)$ be such that $\varphi_l(xH) = gxH$. If $x$ is of the form \eqref{eq Lg S3 3}, then $\Ad(x)\kh = \kh$, so it follows from Lemma \ref{lem phi Ad} and regularity of $g$ that the kernel of $T_{xH}(\varphi_l \circ g^{-1})$ is one-dimensional.

If $x$ is of the form \eqref{eq Lg S3 4}, then $\Ad(x)\kh = \kso(2) \oplus \{0\}$, and a similar argument applies.
\end{proof}

\begin{lemma}\label{lem Gamma g S3}
For all $l \in L_g(\varphi)$, the set $\Gamma_l^g(\varphi)$ consists of a single flow curve. 
\end{lemma}
\begin{proof}
Let $l \in L_g(\varphi)$. Let $x,x' \in \SO(4)$ be such that $\varphi_l(xH) = gxH$ and  $\varphi_l(x'H) = gx'H$.

First , suppose that $l \in \varepsilon \theta_1 + 2\pi \Z$, for $\varepsilon \in \{\pm 1\}$. Then by Lemma \ref{lem Lg S3}, there are $a,d, a', d' \in \SO(2)$ such that
\[
x = \begin{pmatrix} aw_{\varepsilon} & 0 \\ 0 & d w_{\varepsilon}\end{pmatrix}, \qquad 
x' = \begin{pmatrix} a'w_{\varepsilon} & 0 \\ 0 & d' w_{\varepsilon}\end{pmatrix},
\]
so that
\[
x^{-1}x' =  \begin{pmatrix} w_{\varepsilon} a^Ta' w_{\varepsilon} & 0 \\ 0 &  w_{\varepsilon}d^Td' w_{\varepsilon} \end{pmatrix}.
\]
As $w_{\varepsilon} a^T a w_{\varepsilon} \in \SO(2)$, it follows that $x'H = x\exp(tX_0)H$ for some $t \in \R$.

Next, suppose that $l \in \varepsilon \theta_2 + 2\pi \Z$, for $\varepsilon \in \{\pm 1\}$. Then by Lemma \ref{lem Lg S3}, there are $b,c, b', c' \in \SO(2)$ such that
\[
x = \begin{pmatrix} 0 & bw_{\varepsilon}  \\  c w_{-\varepsilon} & 0 \end{pmatrix}, \qquad 
x' = \begin{pmatrix} 0 & b'w_{\varepsilon} \\  c' w_{-\varepsilon}& 0\end{pmatrix},
\]
which implies that
\[
x^{-1}x' =  \begin{pmatrix} w_{-\varepsilon} c^Tc' w_{-\varepsilon} & 0 \\ 0 &  w_{\varepsilon}b^Tb' w_{\varepsilon} \end{pmatrix}.
\]
As in the previous case, it follows that $x'H = x\exp(tX_0)H$ for some $t \in \R$.
\end{proof}


\begin{proposition}\label{prop Ruelle S3}
The flow $\varphi$ is $g$-nondegenerate, the set $L_g(\varphi)$ is countable, and 
 for all $\sigma \in \C$ with positive real parts, 
\beq{eq Ruelle S3}
 \log R^g_{\varphi, \nabla^F}(\sigma) = 
2 \pi \sum_{j=1}^2 \sum_{k=0}^1 \sum_{n \in \Z} \frac{e^{  -|(-1)^k\theta_j + 2\pi n| \sigma}}{|(-1)^k\theta_j + 2\pi n|} 
\eeq
\end{proposition}
\begin{proof}
As in the proof of Proposition \ref{prop Ruelle S2},  $\sgn(\det(1-P_{\gamma}^{g})) = 1$. We can choose $\chi$ to be the constant $1$, so that  
 \[
\int_{I_{\gamma}} \chi(x\gamma(s))\, ds = 2\pi
\]
 for all $l \in L_g(\varphi)$, $\gamma \in \Gamma_g^l(\varphi)$ and 
$x \in \SO(4)$.
Using Lemmas 
  \ref{lem Lg S3}, \ref{lem Gamma g S3}, we obtain \eqref{eq Ruelle S3}.
\end{proof}
%
%

As in the case $n=2$, the equivariant Ruelle $\zeta$-function does not have a meromorphic continuation to $\sigma=0$ in this case, and the kernel of  $\Delta_F$ is nonzero. This implies that the hypotheses of the equivariant Fried conjecture are not satisfied.


%
%


 \bibliographystyle{plain}

\bibliography{mybib}

\end{document}